\documentclass[12pt,reqno]{amsart}
\usepackage[notref,notcite]{}
\usepackage{graphicx}
\usepackage{amssymb}
\usepackage{amsmath}
\usepackage{amsmath,color}

\newcommand{\R}{\mathbb{R}}
\newcommand{\D}{\mathbb{D}}

\begin{document}


\title[Calder\'{o}n problem and flat metrics]{Two-dimensional Calder\'{o}n problem and flat metrics}
\author{Vladimir A. Sharafutdinov}
\address{Sobolev Institute of mathematics. 4 Koptyug av., Novosibirsk, 630090, Russia}
\email{sharafut@list.ru}
\thanks{The work was performed according to the Government research assignment
for IM SB RAS, project FWNF-2022-0006.}

\setcounter{section}{0}
\setcounter{page}{1}
\newtheorem{theorem}{Theorem}[section]
\newtheorem{lemma}[theorem]{Lemma}
\newtheorem{problem}[theorem]{Problem}
\newtheorem{proposition}[theorem]{Proposition}
\newtheorem{corollary}[theorem]{Corollary}
\newtheorem{conjecture}[theorem]{Conjecture}
\newtheorem{definition}[theorem]{Definition}
\newtheorem{remark}[theorem]{Remark}
\numberwithin{equation}{section}

\maketitle


\begin{abstract}
For a compact Riemannian manifold $(M,g)$ with boundary $\partial M$, the Diri\-chl\-et-to-Neumann operator
$\Lambda_g:C^\infty(\partial M)\longrightarrow C^\infty(\partial M)$ is defined by
$\Lambda_gf=\left.\frac{\partial u}{\partial\nu}\right|_{\partial M}$,
where $\nu$ is the unit outer normal vector to the boundary and $u$ is the solution to the Dirichlet problem
$\Delta_gu=0,\ u|_{\partial M}=f$.
Let $g_\partial$ be the Riemannian metric on $\partial M$ induced by $g$.
The Calder\'{o}n problem is posed as follows: To what extent is $(M,g)$ determined by the data $(\partial M,g_\partial,\Lambda_g)$?
We prove the uniqueness theorem: A compact connected two-dimensional Riemannian manifold $(M,g)$ with non-empty boundary is determined by the data $(\partial M,g_\partial,\Lambda_g)$ uniquely up to conformal equivalence.
\end{abstract}

\medskip



\section{Introduction}

Let $(M,g)$ be a compact connected Riemannian manifold with non-empty boundary $\partial M$. We involve the requirement of smoothness of
$M,\ \partial M$ and $g$ into the definition of a Riemannian manifold; the term ``smooth'' is used as a synonym of ``$C^\infty$-smooth''. Let $g_\partial$ be the Riemannian metric on $\partial M$ induced by $g$.
Let $\Delta_g:C^\infty(M)\to C^\infty(M)$ be the Laplace -- Beltrami operator of the metric $g$. It is expressed in local coordinates by
\begin{equation}
\Delta_gu=\frac{1}{\sqrt{\mbox{det}\,g}}
\sum\limits_{i,j=1}^n\frac{\partial}{\partial x^i}\left({\sqrt{\mbox{det}\,g}}g^{ij}\frac{\partial u}{\partial x^j}\right),
                                \label{1.1}
\end{equation}
where $(g^{ij})=(g_{ij})^{-1}$ and $\mbox{det}\,g=\mbox{det}\,(g_{ij})$.

The Dirichlet-to-Neumann operator (DN-map)
\begin{equation}
\Lambda_g:C^\infty(\partial M)\longrightarrow C^\infty(\partial M)
                                \label{1.2}
\end{equation}
is defined by
$$
\Lambda_gf=\left.\frac{\partial u}{\partial\nu}\right|_{\partial M},
$$
where $\nu$ is the unit outer normal vector to the boundary and $u$ is the solution to the boundary value problem
\begin{equation}
\Delta_gu=0\quad\mbox{in}\quad M,\quad u|_{\partial M}=f.
                                \label{1.3}
\end{equation}

Let us now pose the inverse problem. The Riemannian manifold
$(\partial M,g_\partial)$ and operator $\Lambda_g$ are assumed to be known. One has to recover $(M,g)$ from the data
$(\partial M,g_\partial,\Lambda_g)$.
This problem is named {\it the geometric problem of electric impedance tomography} in \cite{Sr2} and {\it the Calder\'{o}n problem} in \cite{Be}. We use the second shorter name in the present paper.

The condition on connectedness of $M$ is necessary in the Calder\'{o}n problem. Indeed, otherwise one of connected components of $M$ can be a manifold with no boundary. The data $(\partial M,g_\partial,\Lambda_g)$ contain no information about such a component.

The following non-uniqueness in the Calder\'{o}n problem is obvious. If $\varphi:M\rightarrow M$ is a diffeomorphism of $M$ onto itself fixing the boundary, $\varphi|_{\partial M}=Id$, then the metric $g'=\varphi^*g$ satisfies $g'_{\!\partial}=g_{\partial}$ and $\Lambda_{g'}=\Lambda_g$. The equality $g'=\varphi^*g$ means that
$\langle v,w\rangle_{g'}=\langle (d_p\varphi)v,(d_p\varphi)w\rangle_{g}$
for any point $p\in M$ and any vectors $v$ and $w$ belonging to the tangent space $T_pM$. Hereafter $\langle \cdot,\cdot\rangle_{g}$ stands for the scalar product of tangent vectors with respect to the metric $g$ and $d_p\varphi:T_pM\rightarrow T_{\varphi(p)}M$ is the differential of  $\varphi$ at $p$. Observe that $\varphi:(M,g')\rightarrow (M,g)$ is an isometry of Riemannian manifolds, therefore the non-uniqueness is clear from the geometric viewpoint.

There exists the conjecture: the ambiguity of the previous paragraph exhausts the non-uniqueness of the Calder\'{o}n problem in dimensions
$\geq 3$, i.e., a compact connected Riemannian manifold of dimension $\geq 3$ with non-empty boundary is determined by the data
$(\partial M,g_\partial,\Lambda_g)$ uniquely up to an isometry fixing the boundary. We emphasize that the statement is a conjecture so far, it is proved for real analytic manifolds only \cite{LaU}.

In the two-dimensional case, the Calder\'{o}n problem possesses one more non-uniqueness. The Laplace -- Beltrami operator on a two-dimensional Riemannian manifold is conformally invariant in the following sense: $\Delta_{\rho g}=\rho^{-1}\Delta_g$ for a positive function
$\rho\in C^\infty(M)$. If the function satisfies the boundary condition $\rho|_{\partial M}=1$, then $\Lambda_{\rho g}=\Lambda_g$.

In what follows, we discuss the two-dimensional Calder\'{o}n problem only in the present paper. Let us agree to use the name {\it a metric surface} for a connected two-dimensional Riemannian manifold. The standard name for the latter object is ``a Riemannian surface'', but it is very close to the term ``a Riemann surface'' that is used in another sense in complex analysis. Moreover, Riemann surfaces will participate in our arguments starting with Section 5. Therefore we do not use the term ``Riemannian surface'' in the present paper.

The boundary of a metric surface $(M,g)$ will be mostly denoted by $\Gamma=\partial M$. Instead of $g_\partial$, we use the arc length $ds_g$ of the curve $\Gamma$ with respect to the metric $g$. For a compact metric surface $(M,g)$ with non-empty boundary $\Gamma$, the DN-map
$$
\Lambda_g:C^\infty(\Gamma)\to C^\infty(\Gamma)
$$
is a non-negative self-adjoint operator with respect to the $L^2$-product
$$
(u,v)_{L^2(\Gamma)}=\int_\Gamma u\overline v\,ds_g\quad(u,v\in C^\infty(\Gamma)).
$$
The one-dimensional kernel of $\Lambda_g$ consists of constant functions while the range of $\Lambda_g$ coincides with the space
$$
C_0^\infty(\Gamma)=\Big\{f\in C^\infty(\Gamma)\mid\int_\Gamma f\,ds_g=0\Big\}
$$
of functions with zero mean value.

For a smooth map $\varphi:N\rightarrow N'$ between two manifolds, we define the linear operator
$\varphi^*:C^\infty(N')\rightarrow C^\infty(N)$ by $\varphi^* u=u\circ\varphi$.
Two ambiguities mentioned above exhaust the non-uniqueness in the two-dimensional Calder\'{o}n problem.

\begin{theorem} \label{Th1.1}
Let $(M_j,g_j)\ (j=1,2)$ be two compact metric surfaces with non-empty boundaries $\Gamma_j=\partial M_j$ and let
$\varphi:(\Gamma_1,ds_{g_1})\rightarrow(\Gamma_2,ds_{g_2})$ be an isometry preserving the DN-map, i.e., such that the following diagram is commutative:
$$
\begin{array}{ccc}
C^\infty(\Gamma_1)&\stackrel{\varphi^*}\longleftarrow&C^\infty(\Gamma_2)\\
\Lambda_{g_1}\downarrow&&\downarrow\Lambda_{g_2}\\
C^\infty(\Gamma_1)&\stackrel{\varphi^*}\longleftarrow&C^\infty(\Gamma_2).
\end{array}
$$
Then $\varphi$ extends to a diffeomorphism $\psi:M_1\rightarrow M_2$ such that $\psi|_{\Gamma_1}=\varphi$ and $\psi^*g_2=\rho g_1$ for some function $0<\rho\in C^\infty(M_1)$ satisfying $\rho|_{\partial M_1}=1$.
\end{theorem}

There exist at least two proofs of the theorem by Lassas -- Uhlmann \cite{LaU} and by Belishev \cite{Be}. The first proof is based on the following observation. The DN-map allows us to determine the Green function $G(x,y)$ for $x,y\in\tilde M\setminus M$, where $\tilde M$ is obtained by gluing a collar to $M$. Starting with the Green function on $\tilde M\setminus M$, one constructs some analytic sheaf whose linear connection component can be identified with $M$. These arguments prove the stronger version of Theorem \ref{Th1.1}: roughly speaking, the knowledge of the DN-map on any open subset of $\Gamma$ is sufficient for recovering $(M,g)$; see \cite{LaU} for the exact statement.
Nevertheless, I am not completely satisfied with Lassas -- Uhlmann's proof.
Indeed, the analytic continuation along curves is used in the proof. If the surface $M$ is not simply connected, the continuation can give a multi-valued function. The difficulty is not discussed in \cite{LaU}.

Belishev's proof is based on the well known Gelfand theorem: a compact topological space $X$ is uniquely, up to a homeomorphism, determined by the Banach algebra $C(X)$ of continuous functions. In the case of a complex analytic manifold $X$, the same statement is true for the Banach algebra ${\mathcal A}(X)$ of holomorphic functions. Belishev observes that, under hypotheses of Theorem \ref{Th1.1}, the DN-map allows us to construct a Banach algebra isomorphic to ${\mathcal A}(M)$. In my opinion, there are serious gaps in Belishev's proof. First of all Theorem \ref{Th1.1} is formulated in the general case but is proved in \cite{Be} only for surfaces with connected boundaries. Some other gaps in Belishev's proof are mentioned below.

An elementary proof of Theorem \ref{Th1.1} for simply connected surfaces is presented in \cite{Sr2}. The proof is based on the following almost obvious fact: for a compact metric surface $(M,g)$ with non-empty boundary, there exists a function $0<\rho\in C^\infty(M)$ such that $\rho|_{\partial M}=1$ and $\rho g$ is a flat metric, i.e., its Gaussian curvature is identically equal to zero. This implies, in the case of a simply connected $M$, that $(M,\rho g)$ can be isometrically immersed into the Euclidean plane. This reduces Theorem \ref{Th1.1} to the partial case when both surfaces $M_1$ and $M_2$ are simply connected, probably multi-sheet, planar domains and metrics $g_1$ and $g_2$ coincide with the standard Euclidean metric of ${\R}^2$. In the latter case, the theorem is easily proved by using basic properties of conformal maps.

A similar result is obtained by Henkin -- Michel \cite{HM}. Roughly speaking, they consider metric surfaces $(M,g)$ such that the boundary curve $\Gamma$ has singular points like angles, but must be a real analytic curve between singular points. The article is overloaded by technical details related to singular points. It is not easy to understand the main idea of their approach, if such an idea exists.

In the preset paper we combine Belishev's approach \cite{Be} with the approach of \cite{Sr2} to get a complete rigorous (I hope) proof of Theorem \ref{Th1.1} in the general case, with the only minor exception: both $M_1$ and $M_2$ are assumed to be oriented surfaces and the isometry $\varphi:(\Gamma_1,ds_{g_1})\rightarrow(\Gamma_2,ds_{g_2})$ is assumed to preserve the induced orientations of boundaries. The case of non-orientable surfaces is reduced to the oriented case in \cite{BK}.

The rest of the paper is organized as follows. In Section 2 we prove that topology of a compact metric surface is determined by the DN-map.
Flat metric surfaces are defined in Section 3; for completeness, we reproduce the proof of existence of a flat metric in a conformal equivalence class. In Section 4 we study conformal maps of flat metric surfaces, Proposition \ref{P4.1} plays a crucial role in our proof of Theorem \ref{Th1.1}. The inner part $\stackrel\circ M=M\setminus\partial M$ of a flat metric surface is furnished with the so called flat complex structure, such structures are discussed in Section 5. In Section 6 we prove that, for a flat metric surface, the Banach algebra
${\mathcal A}(M,g)$ of holomorphic functions is determined, up to an isomorphism, by the DN-map $\Lambda_g$. As mentioned above, this fact was proved by Belishev in the case of connected curve $\Gamma=\partial M$; the reader will see that the general case is more complicated. The Gelfand transform of the algebra ${\mathcal A}(M,g)$ is discussed in Section 7. Finally, the proof of Theorem \ref{Th1.1} is presented in Section 8, it turns out to be an easy corollary of Proposition \ref{P4.1} and Lemma \ref{L8.1}.

I am grateful to Constantin Storozhuk and Segey Treil for useful discussions of the subject. In particular, the presented proof of Lemma \ref{L5.1} belongs to Storozhuk; my initial proof was longer.

\section{Recovering topology from the DN-map}

In this section we prove that a compact oriented metric surface with non-empty boundary is uniquely, up to a diffeomorphism, determined by its DN-map.

A compact connected orientable two-dimensional manifold $M$ is uniquely, up to a diffeomorphism, determined by two integers: the first Betti number $\beta_1(M)$ and the amount $m$ of connected components of the boundary. These integers satisfy
$\beta_1(M)\ge m-1$.
In the setting of the Calder\'{o}n problem, $m$ is given. The only question is to find $\beta_1(M)$.

In \cite{BS}, the DN-map is generalized to exterior differential forms on a compact Riemannian manifold of an arbitrary dimension $n$. One of the main results of \cite{BS} sounds as follows: all Betti numbers of the manifold are determined by the DN-map on forms. We are going to use this statement in the case of $n=2$. For convenience, we shortly reproduce basic facts of \cite{BS}.

Let $(M,g)$ be a compact oriented $n$-dimensional Riemannian manifold with boundary.
Let
$ \Omega^k(M) $ be the space of smooth exterior differential forms of degree $ k $ and $ \Omega(M)=\oplus_{k=0}^n\Omega^k(M) $,
the graded algebra of all forms. Unlike \cite{BS}, $ \Omega^k(M) $ is now a complex vector space, i.e., we consider forms with complex coefficients; all facts presented below are valid for such forms.
We use the following standard operators on $ \Omega(M) $:
the differential $ d $, codifferential $ \delta $,
Laplacian $ \Delta=d\delta+\delta d $, and Hodge star $ \star $.
Recall the relations
$$
\star\star=(-1)^{k(n-k)},\quad
\star\delta=(-1)^kd\star,\quad
\star d=(-1)^{k+1}\delta\star
\quad \mbox{on}\quad \Omega^k(M).
$$

The $ L^2 $-scalar product on $ \Omega(M) $ is defined by
$ (\alpha,\beta)=\int_{M}\alpha\wedge\star\overline\beta $
under the agreement $ \int_{M}\varphi=0 $ for a form
$ \varphi\in \Omega^k(M) $ with $ k<n $.
Recall the Green formula
$$
(d\alpha,\beta)-(\alpha,\delta\beta)=
\int_{\partial M}i^*(\alpha\wedge\star\overline\beta),
$$
where $ i:\partial M\rightarrow M $ is the identical embedding.
For $ \alpha\in\Omega(M) $, the form $ i^*\alpha $ is sometimes called {\it the boundary trace} of $ \alpha $. Nevertheless it is different of the restriction $\alpha|_{\partial M}$ that is the section of the vector bundle $\Omega(M)|_{\partial M}$.

Elements of the space
$$
{\mathcal H}^k(M)=
\{\lambda\in\Omega^k(M)\mid d\lambda=0,\,\,\,\delta\lambda=0\}
$$
are called {\it harmonic fields} while the term {\it harmonic forms} is reserved for forms satisfying $\Delta\lambda=0$. These objects are very different for manifolds with boundary. The $ L^2 $-orthogonal {\it Hodge -- Morrey decomposition} holds
$$
\Omega^k(M)=
{\mathcal E}^k(M)\oplus{\mathcal C}^k(M)\oplus{\mathcal H}^k(M).
$$
Here
$$
{\mathcal E}^k(M)=
\{d\alpha\mid \alpha\in\Omega^{k-1}(M),\ i^*\alpha=0\}
$$
and
$$
{\mathcal C}^k(M)=
\{\delta\alpha\mid \alpha\in\Omega^{k+1}(M),\ i^*(\star\alpha)=0\}.
$$

In the space
$ {\mathcal H}^k(M) $,
two finite-dimensional subspaces are distinguished
$$
{\mathcal H}^k_D(M)=
\{\lambda\in{\mathcal H}^k(M)\mid i^*\lambda=0\},
$$
$$
{\mathcal H}^k_N(M)=
\{\lambda\in{\mathcal H}^k(M)\mid i^*(\star\lambda)=0\}
$$
whose elements are called {\it Dirichlet harmonic fields} and {\it Neumann harmonic fields} respectively. Their dimensions are
$$
\mbox{dim}\,{\mathcal H}^k_N(M)
=
\mbox{dim}\,{\mathcal H}^{n-k}_D(M)
=
\beta_k(M),
$$
where $ \beta_k(M) $ is the $ k $th Betti number.
Two $ L^2 $-orthogonal {\it Fridrichs decompositions} hold
\begin{equation}
{\mathcal H}^k(M)={\mathcal H}^k_D(M)
\oplus
{\mathcal H}^k_{{co}}(M),\quad
{\mathcal H}^k(M)={\mathcal H}^k_N(M)
\oplus
{\mathcal H}^k_{{ex}}(M).
                                \label{2.1}
\end{equation}
Here
$$
{\mathcal H}^k_{{ex}}(M)
=
\{\lambda\in{\mathcal H}^k(M)\mid \lambda=d\alpha\},
\quad
{\mathcal H}^k_{{co}}(M)
=
\{\lambda\in{\mathcal H}^k(M)\mid \lambda=\delta\alpha\}
$$
are spaces of exact harmonic fields and of coexact harmonic fields.
The operator $ \star $ maps
$ {\mathcal H}^k_D(M) $ isomorphically onto
$ {\mathcal H}^{n-k}_N(M) $.
Let us also introduce the spaces of boundary traces
$$
i^*{\mathcal H}^{k}(M)=
\{i^*\lambda\mid\lambda\in{\mathcal H}^{k}(M)\},
\quad
i^*{\mathcal H}^{k}_N(M)=
\{i^*\lambda_N\mid\lambda_N\in{\mathcal H}^{k}_N(M)\}.
$$
A Neumann harmonic field $ \lambda_N $ is uniquely determined by its boundary trace $ i^*\lambda_N $.
Therefore the dimension of $ i^*{\mathcal H}^{k}_N(M) $
is equal to
$ \beta_k(M) $.

For any $ 0\leq k\leq n-1 $, the DN operator
\begin{equation}
\Lambda :\Omega^{k}(\partial M)
\rightarrow \Omega^{n-k-1}(\partial M)
                                \label{2.2}
\end{equation}
is defined as follows. Given $ \varphi\in\Omega^{k}(\partial M) $,
the boundary value problem
\begin{equation}
\left\{
\begin{array}{l}
\Delta\omega=0,\\
i^*\omega=\varphi,\quad i^*(\delta\omega)=0
\end{array}
\right.
                                \label{2.3}
\end{equation}
is solvable, see Lemma 3.4.7 of \cite{Sch}. The solution
$ \omega\in\Omega^{k}(M) $ is unique up to an arbitrary Dirichlet
harmonic field $ \lambda_D\in{\mathcal H}^k_D(M) $. Therefore the form
\begin{equation}
\Lambda\varphi
= i^*(\star d\omega)=(-1)^{k+1}i^*(\delta\star\omega)
                                \label{2.4}
\end{equation}
is independent of the choice of the solution $ \omega $ and
$ \Lambda $ is a well defined operator.

In order to distinguish operators \eqref{1.2} and \eqref{2.2}, the latter operator is denoted by $\Lambda$ with no index although it depends on the metric $g$ as well.

In the scalar case of $ k=0 $, the operators \eqref{1.2} and \eqref{2.2} are expressed through each other. Indeed, in this case the boundary value problem \eqref{1.3} coincides with \eqref{2.3} where $\omega$ is replaced with $u$ and $\varphi$ is replaced with $f$,
and definition \eqref{2.4} gives
\begin{equation}
\Lambda f=\frac {\partial f} {\partial\nu}\mu_\partial=(\Lambda_g f)\mu_\partial
\quad \quad (f\in\Omega^0(\partial M)=C^\infty(M)),
                                \label{2.5}
\end{equation}
where
$ \mu_\partial\in\Omega^{n-1}(\partial M) $ is the boundary volume
form. Thus, in the case of $ k=0 $,
the operator $ \Lambda $ differs from
$ \Lambda_g $ by the presence of the factor
$ \mu_\partial $.

By Corollary 3.4 of \cite{BS}, the operator
\begin{equation}
d\Lambda^{-1}d :\Omega^{k}(\partial M)
\rightarrow \Omega^{n-k-1}(\partial M)
                                \label{2.6}
\end{equation}
is well defined. Therefore the operator
$$
\Lambda+(-1)^{kn+k+n}d\Lambda^{-1} d:
\Omega^k(\partial M)\rightarrow\Omega^{n-k-1}(\partial M)
$$
is well defined too. By Theorem 4.2 of \cite{BS}, the range of the latter operator coincides with the space $i^*{\mathcal H}^{n-k-1}_N(M)$ of boundary traces of Nuemann harmonic fields. Hence
\begin{equation}
\mbox{dim}\,\mbox{Ran}
\left[
\Lambda+(-1)^{kn+k+n}d\Lambda^{-1} d
\right]=
\beta_{n-k-1}(M).
                                \label{2.7}
\end{equation}

\bigskip

We return to considering the two-dimensional Calder\'{o}n problem. Let $(M,g)$ be a compact oriented metric surface with non-empty boundary. The boundary $\Gamma=\partial M$ consists of $m\ge1$ connected components
\begin{equation}
\Gamma=\Gamma_1\cup\dots\cup\Gamma_m,\quad L_j=\mbox{length}(\Gamma_j).
                                \label{2.8}
\end{equation}
Each $\Gamma_i$ is diffeomorphic to ${\mathbb S}^1$. We call $\Gamma_j$ {\it boundary circles}.
Since $M$ is oriented, every circle $\Gamma_j$ is also oriented.
Let
$
D=D_s=-i\frac{d}{ds}:C^\infty(\Gamma)\to C^\infty(\Gamma)
$
be the differentiation with respect to the arc length $ds=ds_g$ in the positive direction. The 1-form $ds$ on $\Gamma$ coincides, in the case of $n=2$, with the $(n-1)$-form $\mu_\partial$ used in \eqref{2.5}. The triple $(\Gamma,ds,\Lambda_g)$ is the data in the two-dimensional Calder\'{o}n problem.

For $k=1$, the Fridrichs decompositions \eqref{2.1} look as follows in the two-dimensional case:
\begin{equation}
{\mathcal H}^1(M)={\mathcal H}^1_N(M)\oplus{\mathcal H}^1_{{ex}}(M),
                                \label{2.9}
\end{equation}
\begin{equation}
{\mathcal H}^1(M)={\mathcal H}^1_D(M)\oplus{\mathcal H}^1_{{co}}(M).
                                \label{2.10}
\end{equation}
The Hodge star transforms decompositions \eqref{2.9} and \eqref{2.10} to each other.

\begin{theorem} \label{Th2.1}
Let $(M,g)$ be a compact oriented metric surface with non-empty boundary $\Gamma=\partial M$ and let $i:\Gamma\subset M$ be the identical embedding. The operator
\begin{equation}
D\Lambda_g^{-1} D:C^\infty(\Gamma)\to C^\infty(\Gamma)
                                \label{2.11}
\end{equation}
is well defined. Under the identification $C^\infty(\Gamma)\to\Omega^1(\Gamma), f\mapsto f\,ds$, the range of the operator $\Lambda_g-D\Lambda_g^{-1} D$ coincides with the space $i^*{\mathcal H}^1_N(M)$ of boundary traces of Neumann harmonic 1-fields. In particular,
$$
\beta_1(M)=\mbox{\rm dim}\,\mbox{\rm Ran}
\left[\Lambda_g-D\Lambda_g^{-1} D\right].
$$
Hence $M$ is determined uniquely, up to a diffeomorphism identical on $\Gamma$, by the data $(\Gamma,ds,\Lambda_g)$.
\end{theorem}

\begin{proof}
We identify the space $\Omega^1(\partial M)$ of 1-forms on $\partial M=\Gamma$ with $C^\infty(\Gamma)$ by $f\,ds\mapsto f$ for $f\in C^\infty(\Gamma)$. By \eqref{2.6}, the operator \eqref{2.11}
is well defined. This fact can be checked independently of \eqref{2.6}. Indeed, given $f\in C^\infty(\Gamma)$, the derivative $Df$
belongs to
\begin{equation}
\mbox{Ran}\,D=\dot C^\infty(\Gamma)=\Big\{f\in C^\infty(\Gamma)\mid\int_{\Gamma_j} f\,ds=0\ (j=1,\dots,m)\Big\},
                                \label{2.12}
\end{equation}
where $\Gamma_j$ are boundary circles, see \eqref{2.8}. Observe that $\dot C^\infty(\Gamma)$ is the subspace of codimension $m$ in
$C^\infty(\Gamma)$. Recall also that
$$
\mbox{Ran}\,\Lambda_g= C^\infty_0(\Gamma)=\Big\{f\in C^\infty(\Gamma)\mid\int_{\Gamma} f\,ds=\sum\limits_{j=1}^m\int_{\Gamma_j} f\,ds=0\Big\}.
$$
Since $\dot C^\infty(\Gamma)\subset C^\infty_0(\Gamma)$, the equation
$\Lambda_g h=Df$ has a solution $h\in C^\infty(\Gamma)$. If $\widetilde h\in C^\infty(\Gamma)$ is another solution to the equation
$\Lambda_g\widetilde h=Df$, then $\Lambda_g(h-\widetilde h)=0$. Since the kernel of $\Lambda_g$ consists of constant functions, we have $h-\widetilde h=\mbox{const}$ and $D(h-\widetilde h)=0$. This means that
$D\Lambda_g^{-1}Df=Dh$ is independent of the choice of a solution to the equation $\Lambda_g h=Df$.

Setting $n=2$ and $k=0$ in \eqref{2.7} and identifying $\Omega^1(\partial M)$ with $C^\infty(\Gamma)$ by $f\,ds\mapsto f$, we arrive to
$
\mbox{\rm Ran}\left[\Lambda_g-D\Lambda_g^{-1}D\right]
=i^*{\mathcal H}^1_N(M).
$
\end{proof}

The operator
$D:C^\infty(\Gamma)\to C^\infty(\Gamma)$
of differentiation with respect to the arc length has the $m$-dimensional kernel consisting of {\it locally constant functions}, i.e., of functions constant on every boundary circle $\Gamma_j$.
By \eqref{2.12}, the range of $D$ is the subspace $\dot C^\infty(\Gamma)$ of codimension $m-1$ in $C^\infty_0(\Gamma)$.
The restriction
\begin{equation}
D: \dot C^\infty(\Gamma)\to \dot C^\infty(\Gamma)
                                \label{2.13}
\end{equation}
is an isomorphism. Let
$
D^{-1} : \dot C^\infty(\Gamma)\to \dot C^\infty(\Gamma)
$
be the inverse operator of \eqref{2.13}.
Thus, for $f\in \dot C^\infty(\Gamma)$, the function $D^{-1}f\in \dot C^\infty(\Gamma)$ is the unique anti-derivative of $f$ with zero mean value on every boundary circle. The following equality holds:
\begin{equation}
D^{-1}Df=f+c\quad\big(f\in C^1(\Gamma)\big),
                                \label{2.14}
\end{equation}
where $c$ is a locally constant function.

If $m>1$ in \eqref{2.8}, then $ \dot C^\infty(\Gamma)$ is the proper subspace of $\mbox{Ran}\,\Lambda_g=C^\infty_0(\Gamma)$ and the product $D^{-1}\Lambda_g$ is not defined. To improve the situation, we represent $C^\infty_0(\Gamma)$ as the direct sum
\begin{equation}
C^\infty_0(\Gamma)=\dot C^\infty(\Gamma)\oplus{\mathbb C}^{m-1}(\Gamma),
                                \label{2.15}
\end{equation}
where
\begin{equation}
{\mathbb C}^{m-1}(\Gamma)=\Big\{c\in C^\infty(\Gamma)\mid c|_{\Gamma_j}=c_j=\mbox{const},\ \int_\Gamma c\,ds=
\sum\limits_{j=1}^m c_jL_j=0\Big\}
                                \label{2.16}
\end{equation}
($L_j$ is the length of $\Gamma_j$), and let
$
P:C^\infty_0(\Gamma)\to\dot C^\infty(\Gamma)
$
be the projection onto the first summand on the right-hand side of \eqref{2.15}. The agrement ${\mathbb C}^0(\Gamma)=0$ is used in \eqref{2.15}.
The product $D^{-1}P\Lambda_g$ is well defined. As follows from \eqref{2.14},
\begin{equation}
D^{-1}Dh=Ph\quad\mbox{for every}\ h\in C^\infty_0(\Gamma).
                                \label{2.17}
\end{equation}

The space ${\mathbb C}^{m-1}(\Gamma)$ can be identified with the intersection
$
{\mathbb C}^{m-1}(\Gamma)\simeq{\mathcal H}^1_D(M)\cap{\mathcal H}^1_{ex}(M).
$
In particular,
$
\mbox{dim}\big[{\mathcal H}^1_D(M)\cap{\mathcal H}^1_{ex}(M)\big]=m-1.
$
Indeed, let $c\in C^\infty(\Gamma)$ be a locally constant function, i.e., $c|_{\Gamma_j}=c_j=\mbox{const}\ (j=1,\dots, m)$.
Let $u\in C^\infty(M)$ be the solution to the Dirichlet problem
$$
\Delta_gu=0,\quad u|_\Gamma=c.
$$
The function $u$ is determined by $c$ uniquely up to an additive constant. Therefore we can assume that $\sum_{j=1}^mc_jL_j=0$, i.e.,
$c\in{\mathbb C}^{m-1}(\Gamma)$. The 1-form $du$ is an exact harmonic 1-field, $du\in{\mathcal H}^1_{ex}(M)$. On the other hand,
$
i^*(du)=d(u|_\Gamma)=dc=0,
$
i.e., $du\in{\mathcal H}^1_D(M)$. Thus, $du\in{\mathcal H}^1_D(M)\cap{\mathcal H}^1_{ex}(M)$.

As is mentioned above,  a Neumann harmonic 1-field $\lambda_N\in{\mathcal H}^1_N(M)$ is uniquely defined by its boundary trace $i^*\lambda_N\in\Omega^1(\Gamma)$. Identifying $\Omega^1(\Gamma)$ with $C^\infty(\Gamma)$ by $f\,ds\mapsto f$ for $f\in C^\infty(\Gamma)$,
we introduce the subspace
\begin{equation}
{\mathcal H}^1_N(\Gamma)=\{\lambda\in C^\infty(\Gamma)\mid \lambda\,ds\in i^*{\mathcal H}^1_N(M)\}
                                \label{2.18}
\end{equation}
of $C^\infty(\Gamma)$. It is isomorphic to ${\mathcal H}^1_N(M)$. In particular,
$
\mbox{dim}\,{\mathcal H}^1_N(\Gamma)=\beta_1(M).
$
Observe also that
\begin{equation}
{\mathcal H}^1_N(\Gamma)\subset C^\infty_0(\Gamma).
                                \label{2.19}
\end{equation}
Indeed, for $\lambda\in{\mathcal H}^1_N(\Gamma)$, the equality $\lambda\,ds=i^*\lambda_N$ holds with some
$\lambda_N\in{\mathcal H}^1_N(M)$ and we obtain with the help of the Stokes theorem
$$
\int\limits_\Gamma\lambda\,ds=\int\limits_{\partial M}i^*\lambda_N
=\int\limits_M\lambda_N=0.
$$

The following statement is a generalization of Lemma 1 of \cite{Be}.

\begin{proposition}
Let $(M,g)$ be a compact oriented metric surface with non-empty boundary $\Gamma=\partial M$ consisting of $m\ge1$ boundary circles.
For every function $f\in C^\infty(\Gamma)$, there exist uniquely determined functions $\lambda\in{\mathcal H}^1_N(\Gamma)$ and
$c\in{\mathbb C}^{m-1}(\Gamma)$ such that
\begin{equation}
\Big(1-(P\Lambda_gD^{-1})^2\Big)Df=P(\Lambda_gD^{-1}P\lambda-\Lambda_gc),
                                \label{2.20}
\end{equation}
where 1 is the identity operator. In particular,
\begin{equation}
\mbox{\rm dim}\Big[\Big(1-(P\Lambda_gD^{-1})^2\Big)\dot C^\infty(\Gamma)\Big]\le \beta_1(M)+m-1.
                                \label{2.21}
\end{equation}
\end{proposition}

\begin{proof}
It suffices to prove \eqref{2.20} for a real function $f\in\dot C^\infty(\Gamma)$.
Let $u\in C^\infty(M)$ be the solution to the Dirichlet problem
$$
\Delta_gu=0,\quad u|_\Gamma=f.
$$
Then $du\in{\mathcal H}^1_{ex}(M)$ and $\star du\in{\mathcal H}^1_{co}(M)$. By the first Fridrichs decomposition \eqref{2.9}, $\star du$ is uniquely represented in the form
\begin{equation}
\ast du=\lambda_N+dv,
                                \label{2.22}
\end{equation}
where $\lambda_N\in{\mathcal H}^1_N(M)$ and $v\in C^\infty(M)$ is a real harmonic function. Set $h=v|_\Gamma$. Both summand on the right-hand side of \eqref{2.22} are uniquely determined by $f$, but the function $v$ is determined up to an additive constant. Using the latter ambiguity, we can assume that $h\in C^\infty_0(\Gamma)$.
Restricting \eqref{2.22} onto $\Gamma$, we have
\begin{equation}
(\star du)|_\Gamma=\lambda_N|_\Gamma+(dv)|_\Gamma.
                                \label{2.23}
\end{equation}

We choose coordinates $(s,t)$ in some neighborhood $U\subset M$ of $\Gamma$ so that $s|_\Gamma$ coincides with the arc length and
$(\nabla t)|_\Gamma=\nu$. Let $u=u(s,t)$ and $v=v(s,t)$ be coordinate expressions of $u$ and $v$ in $U$. Then
$$
(du)|_\Gamma=\frac{\partial u}{\partial s}\Big|_\Gamma\,ds+\frac{\partial u}{\partial t}\Big|_\Gamma\,dt
=i(Df)ds+(\Lambda_gf)\,dt.
$$
Therefore
\begin{equation}
(\ast du)|_\Gamma=(\Lambda_gf)\,ds+i(Df)dt.
                                \label{2.24}
\end{equation}
Quite similarly
$
(dv)|_\Gamma=i(Dh)ds+(\Lambda_gh)\,dt.
$
As far as the first term on the right-hand side of \eqref{2.23} is concerned, we have by the definition \eqref{2.18}
\begin{equation}
\lambda_N|_\Gamma=-\lambda\,ds
                                \label{2.25}
\end{equation}
with a function $\lambda\in{\mathcal H}^1_N(\Gamma)$ uniquely determined by $f$.
Substituting expressions \eqref{2.24}--\eqref{2.25} into \eqref{2.23}, we arrive to the system of two equations
\begin{equation}
\Lambda_gf=iDh-\lambda,
                                \label{2.26}
\end{equation}
\begin{equation}
-iDf=\Lambda_gh.
                                \label{2.27}
\end{equation}

We are going to eliminate the function $h$ from the system \eqref{2.26}--\eqref{2.27}. To this end we first imply the operator $P$ to the equation \eqref{2.26}. Since $PDh=Dh$, the result is as follows:
$
P\Lambda_gf=iDh+P\lambda.
$
Since $f\in\dot C^\infty(\Gamma)$, it satisfies $f=D^{-1}Df$ and the previous equation can be written in the form
\begin{equation}
P\Lambda_gD^{-1}Df=iDh-P\lambda.
                                \label{2.28}
\end{equation}
Recall that $h\in C^\infty_0(\Gamma)$. By \eqref{2.15},
$
h=Ph+ic
$
with $c\in{\mathbb C}^{m-1}(\Gamma)$ uniquely determined by $h$. Together with \eqref{2.17}, this gives
$
h=D^{-1}Dh+ic.
$
Substitute this expression into \eqref{2.27} to obtain
$
-iDf=\Lambda_gD^{-1}Dh+\Lambda_gc.
$
Applying the operator $P$ to this equation and using $PDf=Df$, we get
\begin{equation}
-iDf=P\Lambda_gD^{-1}Dh+iP\Lambda_gc.
                                \label{2.29}
\end{equation}

Finally, expressing $Dh$ from \eqref{2.28} and substituting the expression into \eqref{2.29}, we arrive to \eqref{2.20}.
\end{proof}

{\bf Remark 1.} Can the inequality in \eqref{2.21} be replaced with equality? This important question is not easy in my opinion.

{\bf Remark 2.} In the case of a connected boundary $\Gamma$, i.e., if $m=1$ in \eqref{2.8}, $P$ is the identical operator, $c=\mbox{const}$ in \eqref{2.20}, and $\Lambda_gc=0$. The formula \eqref{2.20} simplifies to the following one:
\begin{equation}
\Big(1-(\Lambda_gD^{-1})^2\Big)Df=\Lambda_gD^{-1}\lambda.
                                \label{2.30}
\end{equation}
This coincides, up to notations, with the formula (1.3) of \cite{Be}. The right-hand side of \eqref{2.30} is well defined in virtue of \eqref{2.19}.

\section{Flat metrics}

A Riemannian metric on a two-dimensional manifold is said to be {\it a flat metric} if its Gaussian curvature is identically equal to zero.

\begin{proposition} \label{P3.1}
Let $(M,g)$ be a compact metric surface with non-empty boundary. There exists a unique positive function $\rho\in C^\infty(M)$
satisfying the boundary condition $\rho|_{\partial M}=1$ and such that $\rho g$ is a flat metric.
\end{proposition}

The proposition is presented in \cite[Lemma 3.1]{Sr2} but without the uniqueness statement. Since the uniqueness is now important, we reproduce the proof here.

\begin{proof}[Proof of Proposition \ref{P3.1}]
In some neighborhood of an arbitrary point of a metric surface, there exist so called {\it isothermal coordinates} such that the first quadratic form is $\lambda(du^2+dv^2)$ where $\lambda=\lambda(u,v)$ is a smooth positive function \cite{Ch}. The Gaussian curvature is expressed in isothermal coordinates as follows:
\begin{equation}
K=-\frac{1}{2\lambda}\,\Delta\ln\lambda,
                                \label{3.1}
\end{equation}
where $\Delta=\frac{\partial^2}{\partial u^2}+\frac{\partial^2}{\partial v^2}$.
As follows from \eqref{1.1}, the Laplace -- Beltrami operator is expressed in isothermal coordinates by
\begin{equation}
\Delta_g=\frac{1}{\lambda}\,\Delta.
                                \label{3.2}
\end{equation}

Let $(M,g)$ be a compact metric surface with non-empty boundary. We choose isothermal coordinates $(u,v)$ in some neighborhood of an arbitrary point of $M$.  The same $(u,v)$ are  isothermal coordinates for the metric $\rho g$ with an arbitrary smooth positive function $\rho$. We look for the function in the form $\rho=e^{2\varphi}$. Let $K$ be the Gaussian curvature of $g$ and $K_\varphi$, the Gaussian curvature of $e^{2\varphi}\,g$.
With the help of \eqref{3.1}--\eqref{3.2}, we see that these functions are related by the equation
$
K_\varphi=e^{-2\varphi}(K-\Delta_g\varphi).
$
Although local coordinates were used in the derivation, the equation makes sense globally on $M$. Thus, $\rho g=e^{2\varphi}g$ is a flat metric if and only if $\Delta_g\varphi=K$. Together with the boundary condition $\rho|_{\partial M}=1$, this leads to the Dirichlet problem
\begin{equation}
\Delta_g\varphi=K\ \mbox{in}\ M,\quad\varphi|_{\partial M}=0.
                                \label{3.3}
\end{equation}
The operator $\Delta_g$ possesses all main properties of the standard Laplacian which guarantee the existence and uniqueness of a solution to the problem \eqref{3.3}.
\end{proof}

It is well known (and can be easily proved, see \cite[Lemma 2.4]{Sr2}) that a flat metric surface $(M,g)$ (either with or without boundary) is locally isometric to the Euclidean plane, i.e., for every point $p\in M$, there exist an open neighborhood $U\subset M$ and smooth injective map
\begin{equation}
I:U\to{\mathbb C}={\R}^2
                                \label{3.4}
\end{equation}
such that $I^*e=g$ for the standard Euclidean metric $e=|dz|^2=dx^2+dy^2$ of the plane.
We say that $(U,I)$ is {\it a flat chart} of the flat metric surface $(M,g)$.
For an inner point $p\in M\setminus\partial M$, \eqref{3.4} is an isometry
$
I:(U,g)\to(V,e)
$
onto an open set $V\subset{\mathbb C}$. For a boundary point $p\in\Gamma=\partial M$, \eqref{3.4} is an isometry
\begin{equation}
I:(U,g)\to(D,e)
                                \label{3.5}
\end{equation}
onto a domain $D\subset{\mathbb C}$ containing a smooth curve $\gamma=I(U\cap\Gamma)\subset\partial D$ on its boundary and such that $D\setminus\gamma$ is an open set. The restriction $I|_{U\cap\Gamma}:U\cap\Gamma\to\gamma$ preserves the arc length.
If additionally the surface $M$ is oriented and the map \eqref{3.4} is assumed to transform the chosen orientation of $M$ to the standard orientation of ${\R}^2$, then the map $I$ is unique up to the composition with a shift and rotation of the plane.

\begin{proposition} \label{P3.2}
Let $(M,g)$ be a compact flat metric surface with a non-empty boundary $\partial M$. There exists a non-compact flat metric surface $(M',g')$ with no boundary such that $(M,g)$ admits an isometric embedding into $(M',g')$.
\end{proposition}

\begin{proof}
For an arbitrary $\varepsilon>0$, we can glue the $\varepsilon$-{\it collar} to $M$, i.e. to set
$$
M_\varepsilon=M\cup\big(\partial M\times[0,\varepsilon)\big),
$$
were a point $p\in\partial M$ is identified with $(p,0)\in\partial M\times[0,\varepsilon)$. As well known, the smooth structure of $M$ can be extended to a smooth structure on $M_\varepsilon$. The flat metric $g$ can be also extended to a flat metric $g'$ on $M'=M_\varepsilon$ for sufficiently small $\varepsilon>0$. In the case of $M=\D$, the latter statement is proved in \cite[Lemma 2.3]{Sr2}. The same proof works in the case of an arbitrary $M$.
\end{proof}

Hypotheses of Theorem \ref{Th1.1} can be slightly simplified. Recall that we are going to prove Theorem \ref{Th1.1} under the additional assumption that $M_1$ and $M_2$ are oriented and the isometry $\varphi:\Gamma_1\to\Gamma_2$ preserves induced orientations. First of all, by Proposition \ref{P3.1}, we can assume both $(M_1,g_1)$ and $(M_2,g_2)$ to be flat Riemannian surfaces. Second, by Theorem \ref{Th2.1}, $M_1$ and $M_2$ are diffeomorphic. Without lost of generality we can assume that $M_1=M_2=M$.

Thus, we have two flat metrics $g_1$ and $g_2$ on a compact connected oriented two-dimensional manifold $M$ with a non-empty boundary $\Gamma=\partial M$ and an orientation preserving diffeomorphism $\varphi:\Gamma\to\Gamma$ satisfying $\varphi^*(ds_{g_2})=ds_{g_1}$ and such that the following diagram is commutative
\begin{equation}
\begin{array}{ccc}
C^\infty(\Gamma)&\stackrel{\varphi^*}\longleftarrow&C^\infty(\Gamma)\\
\Lambda_{g_1}\downarrow&&\downarrow\Lambda_{g_2}\\
C^\infty(\Gamma)&\stackrel{\varphi^*}\longleftarrow&C^\infty(\Gamma).
\end{array}
                                \label{3.6}
\end{equation}
The final simplification is as follows: without lost of generality, we can assume $\varphi$ to be the identity map. To this end we have to use the following statement.

{\it Let $M$ be a compact oriented two-dimensional manifold with boundary. Any orientation preserving diffeomorphism
$\varphi:\partial M\to\partial M$ can be extended to an orientation preserving diffeomorphism
$\chi:M\to M$.}

In my opinion, it is an obvious statement; I omit the proof.

Using the latter statement, we extend $\varphi$ participating in \eqref{3.6} to a differmorhism $\chi$ of $M$ onto itself and replace the metric $g_2$ with $g'_2=\chi^*g_2$. Then $g_1$ and $g'_2$ induce the same arc length on $\Gamma$ and satisfy $\Lambda_{g_1}=\Lambda_{g'_2}$.
Thus, Theorem \ref{Th1.1} is equivalent to the following formally weaker statement

\begin{theorem} \label{Th3.3}
Let $g_1$ and $g_2$ be two flat Riemannian metrics an a compact connected oriented two-dimensional manifold $M$ with non-empty boundary $\Gamma=\partial M$. If $g_1$ and $g_2$ induce the same arc length on $\Gamma$ and satisfy $\Lambda_{g_1}=\Lambda_{g_2}$, then there exists an  isometry $\psi:(M,g_1)\to (M,g_2)$ fixing the boundary, $\psi|_{\partial M}=\mbox{\rm Id}$.
\end{theorem}

\section{Conformal maps of flat metric surfaces}

Let $(M_j,g_j)\ (j=1,2)$ be two metric surfaces. A smooth map $\varphi:M_1\to M_2$ is called {\it a conformal map} if $\varphi^*g_2=\rho g_1$ for a function $0<\rho\in C^\infty(M_1)$.

The following statement plays a crucial role in our proof of Theorem \ref{Th1.1}.

\begin{proposition} \label{P4.1}
Let $g_j\ (j=1,2)$ be two flat metrics on a compact connected oriented two-dimensional manifold $M$ with non-empty boundary
$\Gamma=\partial M$. Assume that these metrics induce the same arc length on $\Gamma$. Let $\psi:M\to M$ be a homeomorhism such that (1) the restriction of $\psi$ to $\Gamma$ is the identical map, and (2) for $\stackrel\circ M=M\setminus\Gamma$, the restriction
$\psi|_{\stackrel\circ M}:(\stackrel\circ M,g_1)\to(\stackrel\circ M,g_2)$
is a conformal map. Then $\psi:(M,g_1)\to(M,g_2)$ is an isometry.
\end{proposition}

{\bf Remark.} Proposition \ref{P4.1} is a generalization of the following well known statement on planar domains:

{\it Let $M_j\subset{\mathbb C}\ (j=1,2)$ be two closed domains bounded by smooth closed curves $\Gamma_j$. If there exists a homeomorhism
$\psi:M_1\to M_2$ such that $\psi|_{{\stackrel\circ M_1}}:{\stackrel\circ M}_1\to{\stackrel\circ M}_2$ is a biholomorphism and the restriction
$\psi|_{\Gamma_1}:\Gamma_1\to\Gamma_2$ preserves the arc length, then $M_1$ and $M_2$ are isometric, i.e., $M_2$ is obtained from $M_1$ by a shift and rotation.}

To prove Proposition \ref{P4.1}, we need the following

\begin{lemma} \label{L4.2}
Under hypotheses of Proposition \ref{P4.1} $\psi$ is a diffeomorphism, i.e., $\psi$ is a $C^\infty$-map and $\psi^{-1}$ is a $C^\infty$-map.
\end{lemma}

We will first demonstrate how Proposition \ref{P4.1} follows from Lemma \ref{L4.2}.

\begin{proof}[Proof of Proposition \ref{P4.1}]
By Lemma \ref{L4.2}, $\psi^*g_2=\rho g_1$ with a function $0<\rho\in C^\infty(M)$ satisfying the boundary condition $\rho|_\Gamma=1$. We have to prove that $\rho\equiv1$. To this end we will demonstrate that $\rho$ cannot take its maximum at an inner point belonging to
$\stackrel\circ M$. The same is true for $1/\rho$, i.e., $\rho$ cannot take its minimum at an inner point. Together with the boundary condition $\rho|_\Gamma=1$, this gives $\rho\equiv1$.

Fix a point $p\in\stackrel\circ M$ and set $q=\psi(p)\in\stackrel\circ M$. Choose a flat chart
$I_2:U_2\to W\subset{\mathbb C}$ of the flat surface $(M,g_2)$ in a neighborhood  of $q$ (see \eqref{3.4}). Then choose a flat chart $I_1:U_1\to V\subset{\mathbb C}$ of the flat surface $(M,g_1)$ in a neighborhood  of $p$ such that $\psi^{-1}(U_2)\subset U_1$. In these charts the conformal map $\psi$ is expressed by $(I_2\circ\psi\circ I_1^{-1})(z)=w(z)$, where $w(z)$ is a holomorphic function on
$V\subset{\mathbb C}$. The function $\rho$ satisfies $(\rho\circ I_1^{-1})(z)=|w'(z)|$ for $z\in V$, where $w'$ is the complex derivative of $w$. The modulus $|w'(z)|$ of the holomorphic function $w'$ cannot take its maximum at a point $z\in V$. In particular, $p$ is not a maximum point of the function $\rho$.
\end{proof}

\begin{lemma} \label{L4.3}
Let $D\subset{\mathbb C}$ be an open set and $\gamma\subset\partial D$ be a $C^\infty$-smooth curve lying on the boundary of $D$. Assume a function $f\in C(D\cup\gamma)$ be such that the restriction $f|_D$ is a holomorhic function and the restriction $f|_\gamma$ is $C^\infty$-smooth. Then, for every inner point $\zeta_0\in\gamma$, there exists a neighborhood $U$ such that all derivatives of $f$ are bounded in $U\cap D$, i.e. for every $n$ there exists a constant $C_n$ such that
\begin{equation}
|f^{(n)}(z)|\le C_n\quad(z\in U\cap D).
                                \label{4.1}
\end{equation}
\end{lemma}

{\bf Remark.} There are many results on boundary regularity of holomorphic functions and of conformal maps in complex analysis. But I have not found a statement implying Lemma \ref{L4.3}.

\begin{proof}[Proof of Lemma \ref{L4.3}]
Let $\zeta(s)\ (A\le s\le B)$ be the parametrization of $\gamma$ by the arc length measured from $\zeta_0$, i.e., $\zeta_0=\zeta(0)$. Choose a sufficiently small $r>0$ such that the curve $\gamma$ intersects the circle $\{z\mid|z-\zeta_0|=2r\}$ at two points $\zeta(a),\zeta(b)$ and the segment $\zeta|_{[a,b]}$ separates the disc $B_{2r}(\zeta_0)=\{z\mid|z-\zeta_0|<2r\}$ to two ``semidiscs'' $B^+_{2r}(\zeta_0)$ and $B^-_{2r}(\zeta_0)$. Besides that $\overline{B^+_{2r}(\zeta_0)}\subset D\cup\gamma$ for a sufficiently small $r>0$. The boundary of the ``upper semidisc'' $B^+_{2r}(\zeta_0)$ consists of the segment $\zeta|_{[a,b]}$ of the curve $\gamma$ and of the arc $\delta$ of the circle $\{z\mid|z-\zeta_0|=2r\}$, see Fig. \ref{Lemma4.2}. We set $U=B_{r}(\zeta_0)=\{z\mid|z-\zeta_0|<r\}$.


\begin{figure}
    \centering
    \includegraphics[width=1.15\linewidth]{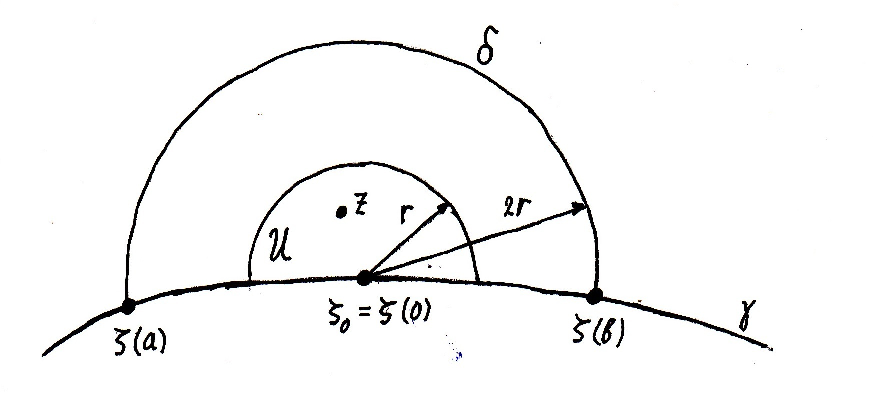}
     \caption{To the proof of Lemma \ref{L4.3}}
    \label{Lemma4.2}
\end{figure}

For $z\in U\cap D$, by the Cauchy formula,
\begin{equation}
f^{(n)}(z)=\frac{n!}{2\pi i}\int\limits_{\gamma|_{[a,b]}}\frac{f(\zeta)\,d\zeta}{(\zeta-z)^{n+1}}
+\frac{n!}{2\pi i}\int\limits_\delta\frac{f(\zeta)\,d\zeta}{(\zeta-z)^{n+1}},
                                \label{4.2}
\end{equation}
where $\gamma|_{[a,b]}$ is the segment of $\gamma$ with the end points $\zeta(a)$ and $\zeta(b)$ and $\delta$ is the arc of the circle $\{z\mid|z-\zeta_0|=2r\}$ with the end points $\zeta(a)$ and $\zeta(b)$.
The second integral on the right-hand side of \eqref{4.2} can be easily estimated by a constant independent of $z$. Indeed, $|\zeta-z|\ge r$ for $\zeta\in\delta$ as is seen from Fig. \ref{Lemma4.2}. Therefore
\begin{equation}
\left|\frac{n!}{2\pi i}\int\limits_\delta\frac{f(\zeta)\,d\zeta}{(\zeta-z)^{n+1}}\right|\le C_n^2
                                \label{4.3}
\end{equation}
with the constant
$$
C_n^2=\frac{n!}{2\pi}\max\limits_{\zeta\in\delta}|f(\zeta)|\,\frac{\mbox{length}(\delta)}{r^{n+1}}\le
\frac{n!}{r^n}\max\limits_{\zeta\in\delta}|f(\zeta)|.
$$

To estimate the first integral on the right-hand side of \eqref{4.2}, we first change the integration variable by $\zeta=\zeta(s)$
$$
\int\limits_{\gamma|_{[a,b]}}\frac{f(\zeta)\,d\zeta}{(\zeta-z)^{n+1}}
=\int\limits_a^b \frac{f(\zeta(s))\zeta'(s)}{(\zeta(s)-z)^{n+1}}\,ds.
$$
By hypotheses of the lemma, the function $\varphi(s)=f(\zeta(s))\zeta'(s)$ belongs to $C^\infty([a,b])$ and does not depend on $z$. The previous formula can be written as
\begin{equation}
\int\limits_{\gamma|_{[a,b]}}\frac{f(\zeta)\,d\zeta}{(\zeta-z)^{n+1}}
=\int\limits_a^b \varphi(s)\,\frac{ds}{(\zeta(s)-z)^{n+1}}.
                                \label{4.4}
\end{equation}
We have to estimate the integral on the right-hand side of \eqref{4.4} by a constant independent of $z$. We will present the estimation in all details for $n=1$, and then just mention the possibility of the same approach for a general $n$.

In the case of $n=1$ the formula looks as follows:
\begin{equation}
\int\limits_{\gamma|_{[a,b]}}\frac{f(\zeta)\,d\zeta}{(\zeta-z)^2}
=\int\limits_a^b \varphi(s)\,\frac{ds}{(\zeta(s)-z)^2}.
                                \label{4.5}
\end{equation}
We notice that $(\zeta(s)-z)^{-2}=-(\zeta'(s))^{-1}\frac{d}{ds}\Big((\zeta(s)-z)^{-1}\Big)$ and transform the integral on the right-hand side of \eqref{4.5} with the help of integration by parts
$$
\begin{aligned}
\int\limits_{\gamma|_{[a,b]}}\frac{f(\zeta)\,d\zeta}{(\zeta-z)^2}
&=-\int\limits_a^b \frac{\varphi(s)}{\zeta'(s)}\,\frac{d}{ds}\Big(\frac{1}{\zeta(s)-z}\Big)\,ds\\
&=\int\limits_a^b \frac{\varphi'(s)\zeta'(s)-\varphi(s)\zeta''(s)}{(\zeta'(s))^2}\,\frac{ds}{\zeta(s)-z}
+\frac{\varphi(a)}{\zeta'(a)}\,\frac{1}{\zeta(a)-z}-\frac{\varphi(b)}{\zeta'(b)}\,\frac{1}{\zeta(b)-z}.
\end{aligned}
$$
We write this in the form
\begin{equation}
\int\limits_{\gamma|_{[a,b]}}\frac{f(\zeta)\,d\zeta}{(\zeta-z)^2}
=\int\limits_a^b \psi(s)\,\frac{ds}{\zeta(s)-z}
+\frac{c_a^1}{\zeta(a)-z}+\frac{c_b^1}{\zeta(b)-z},
                                \label{4.6}
\end{equation}
where
\begin{equation}
\psi(s)=\frac{\varphi'(s)\zeta'(s)-\varphi(s)\zeta''(s)}{(\zeta'(s))^2}
                                \label{4.7}
\end{equation}
and
$$
c_a^1=\frac{\varphi(a)}{\zeta'(a)},\quad c_b^1=-\frac{\varphi(b)}{\zeta'(b)}.
$$
We are not confused by the presence of $\zeta'$ in denominators since $|\zeta'(s)|=1$.

The integral on the right-hand side of \eqref{4.6} can be also transformed with the help of integration by parts. Indeed, since
$(\zeta(s)-z)^{-1}=(\zeta'(s))^{-1}\frac{d}{ds}\big(\ln(\zeta(s)-z)\big)$,
$$
\begin{aligned}
\int\limits_a^b \psi(s)\,&\frac{ds}{\zeta(s)-z}
=\int\limits_a^b \frac{\psi(s)}{\zeta'(s)}\,\frac{d}{ds}\big(\ln(\zeta(s)-z)\big)\,ds\\
&=-\int\limits_a^b\frac{d}{ds}\Big(\frac{\psi(s)}{\zeta'(s)}\Big)\ln(\zeta(s)-z)\,ds
-\frac{\psi(a)}{\zeta'(a)}\,\ln(\zeta(a)-z)+\frac{\psi(b)}{\zeta'(b)}\,\ln(\zeta(b)-z).
\end{aligned}
$$
Substituting this expression into \eqref{4.6}, we write the result in the form
\begin{equation}
\begin{aligned}
\int\limits_{\gamma|_{[a,b]}}\frac{f(\zeta)\,d\zeta}{(\zeta-z)^2}
&=\int\limits_a^b \varphi_1(s)\,\ln(\zeta(s)-z)\,ds\\
&+c_a^0\,\ln(\zeta(a)-z)+\frac{c_a^1}{\zeta(a)-z}+c_b^0\,\ln(\zeta(v)-z)+\frac{c_b^1}{\zeta(b)-z},
\end{aligned}
                                \label{4.8}
\end{equation}
where
\begin{equation}
\varphi_1(s)=-\frac{d}{ds}\Big(\frac{\psi(s)}{\zeta'(s)}\Big)
                                \label{4.9}
\end{equation}
and
$$
c_a^0=-\frac{\psi(a)}{\zeta'(a)},\quad c_b^0=-\frac{\psi(b)}{\zeta'(b)}.
$$

As is seen from Fig. \ref{Lemma4.2},
\begin{equation}
r<|\zeta(a)-z|<3r,\quad r<|\zeta(b)-z|<3r
                                \label{4.10}
\end{equation}
for $z\in U$. On assuming $0<3r<1$, the modulus of each term on the second line of \eqref{4.8} can be estimates by a constant independent of $z$. Thus, the only problem is estimating the integral on the right-hand side of \eqref{4.8}.

As is seen from \eqref{4.7} and \eqref{4.9}, the function $\varphi_1(s)$ is of the form
$$
\varphi_1(s)=\frac{\psi_1(s)}{(\zeta'(s))^4}
$$
with some $\psi_1\in C^\infty([a,b])$ independent of $z$. Since $|\zeta'(s)|=1$, this implies $|\varphi_1(s)|\le C$ with a constant $C$ independent of $z$. Therefore
$$
\left|\int\limits_a^b \varphi_1(s)\,\ln(\zeta(s)-z)\,ds\right|\le
C\int\limits_a^b |\ln(\zeta(s)-z)|\,ds.
$$
Since
$$
|\ln(\zeta(s)-z)|=\Big((\ln|\zeta(s)-z|)^2+(\arg(\zeta(s)-z))^2\Big)^{1/2}\le
\Big((\ln|\zeta(s)-z|)^2+4\pi^2\Big)^{1/2},
$$
this implies
\begin{equation}
\left|\int\limits_a^b \varphi_1(s)\,\ln(\zeta(s)-z)\,ds\right|\le
C\int\limits_a^b \Big((\ln|\zeta(s)-z|)^2+4\pi^2\Big)^{1/2}\,ds.
                                \label{4.11}
\end{equation}

To estimate the integral on the right-hand side of \eqref{4.11}, we are going to change the integration variable by $t=|\zeta(s)-z|$. This needs some justification. First of all $|dt|=ds$ since $|\zeta'(s)|=1$. Let $d(z)>0$ be the distance from $z$ to the curve $\gamma$. When $s$ runs from $a$ to $b$, the variable $t$ changes from $|\zeta(a)-z|$ to $d(z)=|\zeta(s_0)-z|$ for some $s_0\in(a,b)$, and then changes from $d(z)$ to $|\zeta(a)-z|$. As is seen from Fig. \ref{Lemma4.2}, for a sufficiently small $r$, the derivative $dt/ds$ is negative for $s\in[a,s_0]$, and the derivative $dt/ds$ is positive for $s\in[s_0,b]$. Therefore
$$
\int\limits_a^b \Big((\ln|\zeta(s)-z|)^2+4\pi^2\Big)^{1/2}\,ds
=\int\limits_{d(z)}^{|\zeta(a)-z|} \Big((\ln|t|)^2+4\pi^2\Big)^{1/2}\,dt
+\int\limits_{d(z)}^{|\zeta(b)-z|} \Big((\ln|t|)^2+4\pi^2\Big)^{1/2}\,dt.
$$
This implies with the help of \eqref{4.10}
\begin{equation}
\int\limits_a^b \Big((\ln|\zeta(s)-z|)^2+4\pi^2\Big)^{1/2}\,ds
\le\int\limits_0^{3r} \Big((\ln|t|)^2+4\pi^2\Big)^{1/2}\,dt.
                                \label{4.12}
\end{equation}
The integral on the right-hand side of \eqref{4.12} is finite and independent of $z$. Denoting the value of the integral by $c$, we obtain from \eqref{4.11}--\eqref{4.12}
\begin{equation}
\left|\int\limits_a^b \varphi_1(s)\,\ln(\zeta(s)-z)\,ds\right|\le cC.
                                \label{4.13}
\end{equation}

As we have mentioned, the modulus of every term on the second line of \eqref{4.8} can be estimates by a constant independent of $z$. Therefore \eqref{4.8} and \eqref{4.13} give
\begin{equation}
\left|\int\limits_{\gamma|_{[a,b]}}\frac{f(\zeta)\,d\zeta}{(\zeta-z)^2}\right|\le C
                                \label{4.14}
\end{equation}
with some new constant $C$ independent of $z$.

Finally, combining \eqref{4.2}--\eqref{4.3} and \eqref{4.14}, we obtain \eqref{4.1} for $n=1$.

In the case of a general $n$, we transform the first integral on the right-hand side of \eqref{4.2} by applying the integration by parts $n+1$ times. In this way we obtain the following analog of \eqref{4.8}:
$$
\begin{aligned}
\int\limits_{\gamma|_{[a,b]}}&\frac{f(\zeta)\,d\zeta}{(\zeta-z)^{n+1}}
=\int\limits_a^b \varphi_n(s)\,\ln(\zeta(s)-z)\,ds\\
&+c_{a,n}^0\,\ln(\zeta(a)-z)+\sum\limits_{k=1}^n\frac{c_{a,n}^k}{(\zeta(a)-z)^k}+c_{b,n}^0\,\ln(\zeta(v)-z)
+\sum\limits_{k=1}^n\frac{c_{b,n}^k}{(\zeta(b)-z)^k}.
\end{aligned}
$$
All other arguments are the same.
\end{proof}

\begin{proof}[Proof of Lemma \ref{L4.2}]
Let hypotheses of Proposition \ref{P4.1} be satisfied. Being a conformal map, the restriction
$\psi|_{\stackrel\circ M}:\stackrel\circ M\to\stackrel\circ M$ is a smooth map. In particular, the differential $d\psi(p):T_pM\to T_pM$ is well defined for every point $p\in\stackrel\circ M$ and smoothly depends on $p$.

Let $TM$ be the tangent bundle of $M$ and $\otimes_r^sTM$ be the bundle of $(r+s)$-tensors that are $r$ times covariant and $s$ times contravariant ($r$ and $s$ are non-negative integers). Let $C^\infty(\otimes_r^sTM)$ be the space of smooth sections of $\otimes_r^sTM$, and let $\otimes_r^sTM|_{\stackrel\circ M}$ be the restriction of $\otimes_r^sTM$ to $\stackrel\circ M$. The Levi-Chivita connection of the Riemannian manifold $(M,g_1)$ determines the covariant derivative
$\nabla:C^\infty(\otimes_r^sTM)\to C^\infty(\otimes_{r+1}^sTM)$. The higher order derivatives
$\nabla^{(n)}:C^\infty(\otimes_r^sTM)\to C^\infty(\otimes_{r+n}^sTM)$ are also well defined. Finally, $\otimes_r^sTM$ is a Riemannian vector bundle, i.e., the norm $|u|^2\in C^\infty(M)$ is well defined for $u\in C^\infty(\otimes_r^sTM)$.

We consider the differential $d\psi$ as a smooth tensor field on $\stackrel\circ M$, i.e.,
$d\psi\in C^\infty(\otimes_1^1TM|_{\stackrel\circ M})$. We have to prove that it extends to a smooth tensor field on the whole of $M$.
To this end we will first prove that, for every integer $n\ge0$, the norm $|\nabla^{(n)}(d\psi)|$ is bounded on $\stackrel\circ M$. Since $M$ is compact, it suffices to prove the local boundedness: for every point $p\in\Gamma$, there exists a neighborhood $U$ such that the norm $|\nabla^{(n)}(d\psi)|$ is bounded on $U\cap\stackrel\circ M$.

Fix a point $p\in\Gamma$.
For $j=1,2$, let $I_j:(U_j,g_j)\to(D_j,e)$ be a flat chart of the flat surface $(M,g_j)$ in a neighborhood of $p$, see \eqref{3.5}. Recall that the domain $D_j\subset{\mathbb C}$ contains a smooth curve $\gamma_j=I_j(U_j\cap\Gamma)$ on its boundary such that $D_j\setminus\gamma_j$ is an open set and the restriction $I_j|_{U_j\cap\Gamma)}:U_j\cap\Gamma\to\gamma_j$ preserves arc lengths measured in metrics $g_j$ and $e$ respectively. Besides that, the neighborhoods can be chosen so that $\psi^{-1}(U_2)\subset U_1$; it is possible since $\psi|_\Gamma=\mbox{Id}$.

Define the map
$
f=I_2\circ\psi\circ I_1^{-1}:D_1\to D_2\subset{\mathbb C}.
$
It is continuous since $\psi$ is a homeomorphism. The restriction $f|_{D_1\setminus\gamma_1}$ is a holomorphic function since
$\psi|_{\stackrel\circ M}:(\stackrel\circ M,g_1)\to(\stackrel\circ M,g_2)$ is a conformal map while $I_1$ and $I_2$ are isometries.
Finally, the restriction $f|_{\gamma_1}:\gamma_1\to{\mathbb C}$ is $C^\infty$ smooth since $\psi|_\Gamma=\mbox{Id}$ while $I_1$ and $I_2$ are smooth. Thus, the function $f$ satisfies hypotheses of Lemma \ref{L4.3}, where one has to set $D=D_1\setminus\gamma_1$ and $\zeta_0=I_1(p)$. Applying the lemma, we can state that all derivatives of $f$ are bounded in $U\cap(D_1\setminus\gamma_1)$ for some neighborhood $U$ of $I_1(p)$.

Covariant derivatives of $d\psi$ in $U_1$ are expressed through derivatives of $f$ and wise versa. Coordinate representations of $I_1$ and $I_2$ participate in the representations, but we do not worry on this since $I_1$ and $I_2$ are smooth up to the boundary. Thus, for every point $p\in\Gamma$, all covariant derivatives of $d\psi$ are bounded in $U\cap\stackrel\circ M$ for some neighborhood $U$ of $p$. Hence, for every integer $n\ge0$, the norm $|\nabla^{(n)}(d\psi)|$ is bounded on $\stackrel\circ M$.

Recall that, for a smooth vector bundle $\xi$ over a compact manifold and for an integer $k\ge0$, the Sobolev space $H^k(\xi)$ consists of sections whose partial derivatives of order $\le k$ are quadratically integrable. The intersection $\bigcap_{k=0}^\infty H^k(\xi)$ coincides with the space $C^\infty(\xi)$ of smooth sections. For a general vector bundle, the definition of Sobolev spaces needs some specification since partial derivatives are defined locally, in the domain of a local coordinate system. The definition is easier for the tensor bundle $\otimes_r^sTM$ over a Riemannian manifold $(M,g)$: a tensor field $u$ belongs to $H^k(\otimes_r^sTM)$ iff the norms $|\nabla^{(n)}(d\psi)|$ are quadratically integrable over $M$ for all $n\le k$.

Returning to the proof of the lemma, we have shown that the norms $|\nabla^{(n)}(d\psi)|$ are bounded on $\stackrel\circ M$ for all $n$. Since $M$ is compact, all norms $|\nabla^{(n)}(d\psi)|$ are quadratically integrable over $M$. Hence $d\psi\in H^k(\otimes_1^1TM)$ for every $k$, i.e., $d\psi\in C^\infty(\otimes_1^1TM)$. This means that $\psi$ is $C^\infty$-smooth on the whole of $M$. The same is true for $\psi^{-1}$.
\end{proof}

\section{Flat Riemann surfaces}

Probably, some experts in complex analysis know the following statement but I have no reference.

\begin{lemma} \label{L5.1}
Let $D_r=\{z=x+iy\in{\mathbb C}\mid y\ge0,|z|<r\}$ with some $r>0$. Let real functions $u,v\in C^\infty(D_r)$ be harmonic in
${\stackrel\circ D}_r=\{z=x+iy\mid y>0,|z|<r\}$ and satisfy the Couchy -- Riemann equations for $y=0$, i.e.,
\begin{equation}
u'_x(x,0)=v'_y(x,0),\quad u'_y(x,0)=-v'_x(x,0)\quad(|x|<r).
                                \label{5.1}
\end{equation}
Then the functions $u$ and $v$ satisfy the Couchy -- Riemann equations on the whole of $D_r$, i.e., $u+iv$ is a holomorphic in $D_r$ function.
\end{lemma}

\begin{proof}
Since $D_r$ is a simply connected domain, there exists a real harmonic function $\widetilde v\in C^\infty(D_r)$ conjugate to $u$. It suffices to prove that
$v-\widetilde v=\mbox{const}$ in $D_r$. The Cauchy -- Riemann equations
$
u'_x={\widetilde v}'_y,\ u'_y=-{\widetilde v}'_x
$
hold on the whole of $D_r$. Together with \eqref{5.1}, the latter equations give
\begin{equation}
d(v-\tilde v)=0\quad\mbox{at}\quad y=0.
                                \label{5.2}
\end{equation}
Repeating the same arguments, we find a real harmonic function $\widetilde u\in C^\infty(D_r)$ such that  $w=(u-\widetilde u)+i(v-\widetilde v)$ is a holomorphic function in $\mbox{int}\,D_r$. The Cauchy -- Riemann equations
$
(u-\widetilde u)'_x=(v-\widetilde v)'_y,\ (u-\widetilde u)'_y=-(v-\widetilde v)'_x
$
hold on the whole of $D_r$. Together with \eqref{5.2}, the latter equations give $dw=0$ at $y=0$, i.e.,
$
w=\mbox{const}\quad\mbox{at}\quad y=0.
$
By the uniqueness principle for holomorphic functions, $w=\mbox{const}$ on the whole of $D_r$. In particular,
$v-\widetilde v=\mbox{const}$ in $D_r$.
\end{proof}

The Cauchy -- Riemann equations $u'_x=v'_y, u'_y=-v'_x$ make sense globally on a flat oriented metric surface, although partial derivatives are sensible only locally in the domain of a flat chart. In particular, {\it holomorphic functions} are well defined on such a surface. Lemma \ref{L5.1} is generalized as follows.

\begin{lemma} \label{L5.2}
Let $(M,g)$ be a flat oriented metric surface with non-empty boundary $\Gamma=\partial M$. If two real harmonic functions
$u,v\in C^\infty(M)$ satisfy the Cauchy -- Riemann equations on a curve $\gamma\subset\Gamma$ of a positive length, then $u+iv$ is a holomorphic function in $M$.
\end{lemma}

\begin{proof}
Fix a point $p\in\gamma$. We first will demonstrate that the restriction of $u+iv$ to some neighborhood $V\subset M$ of $p$ is a holomorphic function in $V$. To this end we choose a simply connected neighborhood $V$ of $p$ such that $V\cap\Gamma\subset\gamma$ and there exists a biholomorphism $\varphi:V\to D_r$ onto the semi-disk $D_r\subset{\mathbb C}$ like in Lemma \ref{L5.1}; the existence of such a biholomorphism is guaranteed by the Riemann theorem. The functions $\tilde u=u\circ\varphi^{-1}$ and $\tilde v=v\circ\varphi^{-1}$ are harmonic in $D_r$ and satisfy the Cauchy -- Riemann equations at $y=0$. By Lemma \ref{L5.1}, $\tilde u+i\tilde v$ is a holomorphic function in $D_r$. Hence  $u+iv=(\tilde u+i\tilde v)\circ\varphi$ is a holomorphic function in $V$ since the composition of a holomorphic function and conformal map is again a holomorphic function.

Let now $\{(U_j,I_j=z_j=x_j+iy_j)\}_{j\in J}$ be a flat atlas on $M$. For every chart $(U,z)$ of the atlas let $u(x,y)$ and $v(x,y)$ be representations of our harmonic functions $u$ and $v$ in the chart. In the domain $U$, the Cauchy -- Riemann equations can be written in coordinates:
$$
u'_x-v'_y=0,\quad u'_y+v'_x=0.
$$
Left-hand sides of the equations are harmonic in $U$ functions. If we assumed the existence of a non-empty open subset
$V\subset U$ such that the Cauchy -- Riemann equations are valid in $V$, then Cauchy -- Riemann equations would be valid on the whole of $U$ in virtue of the uniqueness principle for harmonic functions. In the latter case, the restriction of $u+iv$ to $U$ is a holomorphic function.

Let $p\in U_{j_0}$ for some chart $(U_{j_0},z_{j_0})$ of the atlas. By arguments of two previous paragraphs, the restriction of $u+iv$ to $U_{j_0}$ is a holomorphic function. Passing from $(U_{j_0},z_{j_0})$ to an arbitrary chart $(U_j,z_j)$ of the atlas with the help of a chart chain with non-empty intersections of domains, we convince ourselves that $u+iv$ is a holomorphic function on the whole of $M$.
\end{proof}

We recall the definition of a Riemann surface, compare with the definition 1.4 on page 3 of \cite{F}.

Let $M$ be a smooth two-dimensional manifold without boundary.
A {\it complex chart} on $M$ is a homeomorphism $ z:U\to V$ of an open subset $U\subset M$ onto an open subset $V\subset{\mathbb C}$.
Two complex charts $ z_j:U_j\to V_j,\ j=1,2$ are said to be {\it holomorphically compatible} if the map
\begin{equation}
 z_2\circ z_1^{-1}: z_1(U_1\cap U_2)\to z_2(U_1\cap U_2)
                                \label{5.3}
\end{equation}
is biholomorphic.

A {\it complex atlas} on $M$ is a system $\mathfrak{A}=\{ z_j:U_j\to V_j,\ j\in J\}$ of charts which are holomorphically compatible and which cover $M$, i.e., $\bigcup_{j\in J}U_j=M$.

Two complex atlases $\mathfrak{A}$ and ${\mathfrak A}'$ on $M$ are {\it analitically equivalent} if every chart of $\mathfrak{A}$ is holomorphically compatible with every chart of ${\mathfrak A}'$.

\begin{definition}
{\rm By a {\it complex structure} on a two-dimensional manifold $M$ without boundary we mean an equivalence class of analytically equivalent atlases on $M$.}
\end{definition}

A complex structure on $M$ can be given by the choice of a complex atlas. Every complex structure ${\mathcal C}$ on $M$ contains a unique maximal atlas $\mathfrak{A}^*$. If $\mathfrak{A}$ is an arbitrary atlas in ${\mathcal C}$, then $\mathfrak{A}^*$ consists of all complex charts on $M$ which are holomorphically compatible with every chart of $\mathfrak{A}$.

\begin{definition}
{\rm A {\it Riemann surface} is a pair $(M,{\mathcal C})$ where $M$ is a connected two-di\-men\-sio\-nal manifold without boundary and ${\mathcal C}$ is a complex structure on $M$.}
\end{definition}

One usually writes $M$ instead of $(M,{\mathcal C})$ whenever it is clear which complex structure ${\mathcal C}$ is meant. Sometimes one also writes $(M,{\mathfrak A})$ where ${\mathfrak A}$ is a representative of ${\mathcal C}$.

{\it Convention.} If $M$ is a Riemann surface, then by a chart on $M$ we always mean a complex chart belonging to the maximal atlas of the complex structure on $M$. Given such a chart $z:U\to V$, we also say that $z\in C^\infty(U)$ is a {\it local complex coordinate} on $M$ with the domain $U\subset M$.

\medskip

{\bf Remark.}
We emphasize that a complex structure has been defined on a two-dimen\-si\-on\-al manifold {\it with no boundary} only. Can the definition be generalized to two-dimensional manifolds with non-empty boundaries? Such a generalization is not not unique if possible. Therefore we never consider  a complex structure on a two-dimensional manifold with non-empty boundary.
Belishev in \cite{Be} discusses a {\it complex structure} on a compact metric surface with non-empty boundary, without giving a precise definition. In my opinion it is his systematic mistake.

\medskip

Next, we are going to introduce a very special case of a complex structure which, nevertheless, will play an important role in our arguments.

Let again $M$ be a smooth two-dimensional manifold without boundary.
Two complex charts $ z_j:U_j\to V_j,\ j=1,2$ on $M$ are said to be {\it linearly compatible} if the map \eqref{5.3} is of the form
$( z_2\circ z_1^{-1})(z)=az+b$ for $z\in  z_1(U_1\cap U_2)$, where $a$ and $b$ are some complex constants and moreover $|a|=1$. Then, following the same scheme, we define a {\it flat complex atlas} on $M$ as a family $\mathfrak{A}=\{ z_j:U_j\to V_j,\ j\in J\}$ of charts which are linearly compatible and which cover $M$. Two flat complex atlases $\mathfrak{A}$ and ${\mathfrak A}'$ on $M$ are {\it linearly equivalent} if every chart of $\mathfrak{A}$ is linearly compatible with every chart of ${\mathfrak A}'$.
By a {\it flat complex structure} on a two-dimensional manifold $M$ without boundary we mean an equivalence class of linearly equivalent flat atlases on $M$. And finally we introduce

\begin{definition}
{\rm A {\it flat Riemann surface} is a pair $(M,{\mathcal C})$ where $M$ is a connected two-di\-men\-sio\-nal manifold without boundary and ${\mathcal C}$ is a flat complex structure on $M$.}
\end{definition}

Flat complex structures are closely related to flat Riemannian metrics. We complete the section with the discussion of the relation.

Let $(M,{\mathcal C})$ be a flat Riemann surface. Given a complex chart $z=x+iy:U\to{\mathbb C}$ on $M$, the flat Riemannian metric $|dz|^2=dx^2+dy^2$ is defined on the domain $U\subset M$ of the complex coordinate $z$. If $z':U'\to{\mathbb C}$ is another complex chart on $M$, then $z'=az+b$ on $U\cap U'$ with some constants $a$ and $b$ such that $|a|=1$. This implies $|dz|^2=|dz'|^2$ on $U\cap U'$. Thus, the flat Riemannian metric $|dz|^2$ is globally defined on $M$. This metric  will be sometimes denoted by $|dz|^2=g_{\mathcal C}$ to emphasize that it is determined by the flat complex structure ${\mathcal C}$. Thus, every flat Riemann surface $(M,{\mathcal C})$ can be considered as a flat metric surface $(M,|dz|^2)=(M,g_{\mathcal C})$ (with no boundary).

Let now $(M,g)$ be a flat oriented metric surface with no boundary. Choose a flat atlas
$\mathfrak{A}=\{ I_j:(U_j,g)\to (V_j,e),\ j\in J\}$ of the flat metric surface $(M,g)$ such that the intersection $U_j\cap U_k$ is connected for all $j.k\in J$ and all $I_j$ preserve the orientation. Then $\mathfrak{A}$ is a flat complex atlas. Indeed, if $U_j\cap U_k\neq\emptyset$, then the map
$$
I_k\circ I_j^{-1}:I_j(U_j\cap U_k)\to I_k(U_j\cap U_k)
$$
is an orientation preserving isometry of planar domains. Hence
$$
(I_k\circ I_j^{-1})(z)=a_{jk}z+b_{jk}\quad (z\in U_j\cap U_k;a_{jk}=\mbox{const},b_{jk}=\mbox{const},|a_{jk}|=1).
$$
The same is true for arbitrary $j,k\in J$ since there exists a sequence $j=j_0,j_1,\dots,j_n=k$ such that $U_j\cap U_{j+1}\neq\emptyset$.
Hence $\mathfrak{A}$ is a flat complex atlas.
Thus, every flat oriented metric surface with no boundary becomes a flat Riemann surface in a canonical way.

Finally we consider a compact flat oriented metric surface $(M,g)$ with non-empty boundary. By Proposition \ref{P3.2}, $(M,g)$ can be isometrically embedded into a flat oriented metric surface $(M',g')$ with no boundary. The latter surface is furnished in a canonical way with the flat complex structure ${\mathcal C}'$ such that $g'=g_{{\mathcal C}'}$.

We have thus proved

\begin{proposition} \label{P5.6}
{\rm (1)} Every flat Riemann surface $(M,{\mathcal C})$ is furnished, in a canonical way, with a flat Riemannian metric $g_{\mathcal C}$ such that $g_{\mathcal C}=|dz|^2$ for any complex chart $z:U\to V$ of $(M,{\mathcal C})$.

{\rm (2)} Every flat oriented metric surface $(M,g)$ without boundary is furnished, in a canonical way, with a flat complex structure ${\mathcal C}$ such that $g=g_{\mathcal C}$.

{\rm (3)} Every flat compact oriented metric surface $(M,g)$ with non-empty boundary admits an isometric embedding into a flat Riemann surface $(M',{\mathcal C}')$ considered with the metric $g'=g_{{\mathcal C}'}$.
\end{proposition}

\section{Banach algebras ${\mathcal A}(M)$ and ${\mathcal A}(\Gamma)$}

Let $(M,g)$ be a compact oriented flat metric surface with non-empty boundary $\Gamma=\partial M$ consisting of $m\ge1$ boundary circles, see \eqref{2.8}.

We denote by ${\mathcal A}(M,g)$ the algebra of continuous complex-valued functions
$f\in C(M)$ such that the restriction of $f$ to the interior $\stackrel\circ M=M\setminus\Gamma$ is a holomorphic function.
Belishev \cite{Be} uses the notation ${\mathcal A}(M)$ for the latter algebra. Our notation ${\mathcal A}(M,g)$ is more precise since the  metric $g$ participates in the definition of a holomorphic function.
Elements of ${\mathcal A}(M,g)$ will be called {\it holomorphic functions} on $(M,g)$.
We also introduce the algebra
$$
{\mathcal A}(\Gamma,g)=\{w|_{\Gamma}\mid w\in{\mathcal A}(M,g)\}
$$
of {\it boundary traces of holomorphic functions}.
Define the norms on ${\mathcal A}(M,g)$ and ${\mathcal A}(\Gamma,g)$ by
$$
\begin{aligned}
\|w\|&=\sup\limits_{z\in M}|w(z)|\quad\mbox{for}\ w\in {\mathcal A}(M,g),\\
\|w\|&=\sup\limits_{z\in\Gamma}|w(z)|\quad\mbox{for}\ w\in {\mathcal A}(\Gamma,g).
\end{aligned}
$$
Then ${\mathcal A}(M,g)$ and ${\mathcal A}(\Gamma,g)$ are  commutative Banach algebras. {\it The trace operator}
$$
{\mathcal A}(M,g)\to{\mathcal A}(\Gamma,g),\quad w\mapsto w|_{\Gamma}
$$
is an isomorphism of Banach algebras preserving the norm, $\|w\|=\|w|_\Gamma\|$, by the modulus maximum principle for holomorphic functions.

Let us denote ${\mathcal A}^\infty(M,g)={\mathcal A}(M,g)\cap C^\infty(M)$. Recall also that the space ${\mathbb C}^{m-1}(\Gamma)$ of locally constant functions is defined by \eqref{2.16}.

\begin{proposition} \label{P6.1}
Let $(M,g)$ be a compact oriented flat metric surface with a non-empty boundary $\partial M=\Gamma$ consisting of $m\ge1$ boundary circles. Then

{\rm (1)} If  $a+ib\in{\mathcal A}^\infty(\Gamma,g)$
for $a,b\in C^\infty(\Gamma)$, then
\begin{equation}
\Lambda_ga\in \dot C^\infty(\Gamma),\quad
\Lambda_gb\in \dot C^\infty(\Gamma)
                                \label{6.1}
\end{equation}
and there exists $c\in{\mathbb C}^{m-1}(\Gamma)$ such that
\begin{equation}
Da-\Lambda_gD^{-1}\Lambda_ga-i\Lambda_gc=0
                                \label{6.2}
\end{equation}
and
\begin{equation}
b=-iD^{-1}\Lambda_ga+c.
                                \label{6.3}
\end{equation}

{\rm (2)} Conversely, if functions $a,b\in C^\infty(\Gamma)$ and $c\in{\mathbb C}^{m-1}(\Gamma)$ satisfy \eqref{6.1}--\eqref{6.3}, then $a+ib\in{\mathcal A}^\infty(\Gamma,g)$.

{\rm (3)} Formulas \eqref{6.1}--\eqref{6.3} imply
$
Db-\Lambda_gD^{-1}\Lambda_gb-i\Lambda_gc'=0
$
and
$
a=iD^{-1}\Lambda_gb-c'
$
with some $c'\in{\mathbb C}^{m-1}(\Gamma)$.
\end{proposition}

\begin{proof}
The statement (3) follows from first two statements in virtue of the fact: if $a+ib\in{\mathcal A}^\infty(\Gamma,g)$ then
$b-ia\in{\mathcal A}^\infty(\Gamma,g)$.

The operators $\Lambda_g,iD,iD^{-1}$ are {\it real operators}, i.e., they transform real functions again to real functions; and they are ${\mathbb C}$-linear operators. Therefore it suffices to prove the statement (1) for real functions $a$ and $b$. Moreover, it suffices to prove the statement (1) for real functions $a,b\in C^\infty_0(\Gamma)$ since constant functions lie in kernels of $\Lambda_g$ and of $D$.

Let two real functions $a,b\in C^\infty_0(\Gamma)$ be such that $a+ib\in{\mathcal A}^\infty(\Gamma,g)$, i.e., $a+ib=(u+iv)|_\Gamma$ for real functions $u,v\in C^\infty(M)$ satisfying the Cauchy -- Riemann equations. On the curve $\Gamma$, the Cauchy -- Riemann equations can be written in the form
\begin{equation}
\frac{\partial u}{\partial\nu}\Big|_\Gamma=\frac{d}{ds}(v|_\Gamma),\quad
\frac{\partial v}{\partial\nu}\Big|_\Gamma=-\frac{d}{ds}(u|_\Gamma).
                                \label{6.4}
\end{equation}
Observe that these equations imply
\begin{equation}
\frac{\partial u}{\partial\nu}\Big|_\Gamma\in \dot C^\infty(\Gamma),\quad
\frac{\partial v}{\partial\nu}\Big|_\Gamma\in \dot C^\infty(\Gamma).
                                \label{6.5}
\end{equation}
The functions $u$ and $v$ are solutions to the Dirichlet problems
\begin{equation}
\Delta_gu=0,\quad u|_\Gamma=a,
                                \label{6.6}
\end{equation}
\begin{equation}
\Delta_gv=0,\quad v|_\Gamma=b.
                                \label{6.7}
\end{equation}
By the definition of the DN-map,
\begin{equation}
\Lambda_ga=\frac{\partial u}{\partial\nu}\Big|_\Gamma,
                                \label{6.8}
\end{equation}
\begin{equation}
\Lambda_gb=\frac{\partial v}{\partial\nu}\Big|_\Gamma.
                                \label{6.9}
\end{equation}
Formulas \eqref{6.5} and \eqref{6.9} imply the validity of \eqref{6.1}.

By \eqref{2.15}, the function $b\in C^\infty_0(\Gamma)$ can be uniquely represented in the form
\begin{equation}
b=Pb+c
                                \label{6.10}
\end{equation}
with a locally constant real function $c\in{\mathbb C}^{m-1}(\Gamma)$. Applying the operator $\Lambda_g$ to \eqref{6.10}, we have
$
\Lambda_gb=\Lambda_gPb+\Lambda_gc.
$
Together with \eqref{6.9}, this gives
\begin{equation}
\Lambda_gPb+\Lambda_gc=\frac{\partial v}{\partial\nu}\Big|_\Gamma.
                                \label{6.11}
\end{equation}
By \eqref{6.4},
$
\frac{\partial v}{\partial\nu}\Big|_\Gamma=-\frac{d}{ds}(u|_\Gamma)=-\frac{da}{ds}=-iDa.
$
Substitute this expression into \eqref{6.11} to obtain
\begin{equation}
iDa+\Lambda_gPb+\Lambda_gc=0.
                                \label{6.12}
\end{equation}
Applying the operator $D^{-1}D$ to the equality \eqref{6.10}, we have
\begin{equation}
D^{-1}Db=D^{-1}DPb.
                                \label{6.13}
\end{equation}
By \eqref{2.17}, $D^{-1}DPb=Pb$. Therefore \eqref{6.13} becomes $Pb=D^{-1}Db$. Substitute this expression into \eqref{6.12}
\begin{equation}
iDa+\Lambda_gD^{-1}Db+\Lambda_gc=0.
                                \label{6.14}
\end{equation}
The first equality in \eqref{6.4} and \eqref{6.8} imply
\begin{equation}
Db=-i\frac{db}{ds}=-i\frac{\partial u}{\partial\nu}\Big|_\Gamma=-i\Lambda_ga.
                                \label{6.15}
\end{equation}
Substituting this expression into \eqref{6.14}, we arrive to the equation \eqref{6.2}.

Recall that the operator $D^{-1}$ is defined on $\dot C^\infty(\Gamma)$. The second term on the left-hand side of \eqref{6.2} is well defined in virtue of \eqref{6.1}.

Applying the operator $D^{-1}$ to the equality \eqref{6.15}, we obtain
\begin{equation}
D^{-1}Db=-iD^{-1}\Lambda_ga.
                                \label{6.16}
\end{equation}
Again, the right-hand side of \eqref{6.16} is well defined in virtue of \eqref{6.1}.
By \eqref{2.17} and \eqref{6.10},
$$
D^{-1}Db=Pb=b-c.
$$
Substituting this expression into \eqref{6.16}, we arrive to \eqref{6.3}. The first statement of Proposition \ref{P6.1} is proved.

Again, it suffices to prove the statement (2) for real functions $a,b\in C^\infty_0(\Gamma)$ and $c\in{\mathbb C}^{m-1}(\Gamma)$. Assume formulas \eqref{6.1}--\eqref{6.3} to be valid. Let $u\in C^\infty(M)$ and $v\in C^\infty(M)$ be solutions to Dirichlet problems \eqref{6.6} and \eqref{6.7} respectively. Then equalities \eqref{6.8}--\eqref{6.9} hold. We are going to prove that $u+iv$ is a holomorphic function.

Comparing \eqref{6.3} and \eqref{6.9}, we see that
\begin{equation}
\frac{\partial v}{\partial\nu}\Big|_\Gamma=-i\Lambda_gD^{-1}\Lambda_ga+\Lambda_gc.
                                \label{6.17}
\end{equation}
Formulas \eqref{6.7}, \eqref{6.3} and \eqref{6.8} imply
$$
\frac{d}{ds}(v|_\Gamma)=\frac{db}{ds}=-i\frac{d}{ds}D^{-1}\Lambda_ga=DD^{-1}\Lambda_ga
=\Lambda_ga=\frac{\partial u}{\partial\nu}\Big|_\Gamma.
$$
Thus,
\begin{equation}
\frac{d}{ds}(v|_\Gamma)=\frac{\partial u}{\partial\nu}\Big|_\Gamma.
                                \label{6.18}
\end{equation}

Next, applying the operator $\frac{d}{ds}$ to the second equality in \eqref{6.6}, we have
$
\frac{d}{ds}(u|_\Gamma)=\frac{da}{ds}.
$
By \eqref{6.2},
$
\frac{da}{ds}=iDa=i\Lambda_gD^{-1}\Lambda_ga-\Lambda_gc.
$
Two last formulas give
$
\frac{d}{ds}(u|_\Gamma)=i\Lambda_gD^{-1}\Lambda_ga-\Lambda_gc.
$
Comparing this with \eqref{6.17}, we obtain
\begin{equation}
\frac{\partial v}{\partial\nu}\Big|_\Gamma=-\frac{d}{ds}(u|_\Gamma).
                                \label{6.19}
\end{equation}

We have thus proved that the harmonic functions $u$ and $v$ solve equations \eqref{6.18} and \eqref{6.19} that coincide with \eqref{6.4}. As we know, the latter equations mean that $u$ and $v$ satisfy Cauchy -- Riemann equations on $\Gamma$. Applying Lemma \ref{L5.2}, we see that $u+iv$ is a holomorphic function.
\end{proof}

Given the data $(\Gamma,ds,\Lambda_g)$ of the Calder\'{o}n problem, we know the spaces $\dot C^\infty(\Gamma)$ and ${\mathbb C}^{m-1}(\Gamma,g)$, the norm $\|\cdot\|_{{\mathcal A}(\Gamma)}$, and the operators $D,D^{-1},\Lambda_g$. By Proposition \ref{P6.1}, we can determine the algebra ${\mathcal A}^\infty(\Gamma,g)$. The Banach algebra ${\mathcal A}(\Gamma,g)$ is the closure of ${\mathcal A}^\infty(\Gamma,g)$ in $C(\Gamma)$. As mentioned at the beginning of the current section, Banach algebras ${\mathcal A}(\Gamma,g)$ and ${\mathcal A}(M,g)$ are isomorphic. We have thus proved

\begin{theorem} \label{Th6.2}
The Banach algebra ${\mathcal A}(M,g)$ of holomorphic functions on a compact oriented flat metric surface $(M,g)$ with non-empty boundary
$\partial M=\Gamma$ is uniquely determined, up to an isomorphism, by the data $(\Gamma,ds,\Lambda_g)$. Moreover, the Banach algebra ${\mathcal A}(\Gamma,g)$ of boundary traces of holomorphic functions is uniquely determined by the data $(\Gamma,ds,\Lambda_g)$.
\end{theorem}

{\bf Remark.}
In the case of a connected $\Gamma$, i.e., when $m=1$ in \eqref{2.8}, Proposition \ref{P6.1} actually coincides with Lemma 1 of \cite{Be}. Indeed, in this case $\dot C^\infty(\Gamma)=C^\infty_0(\Gamma),\ {\mathbb C}^{m-1}(\Gamma)=0$, formulas \eqref{6.1} are valid for any functions $a,b\in C^\infty(\Gamma)$, and equations \eqref{6.2}--\eqref{6.3} look as follows:
$
b=-iD^{-1}\Lambda_ga,\ Da-\Lambda_gD^{-1}\Lambda_ga=0.
$
Besides this, $a=D^{-1}Da$ and the last equation can be written in the form
$
\Big(1-(\Lambda_gD^{-1})^2\Big)Da=0,
$
where 1 is the identity operator. This coincides with the formula (1.5) of \cite{Be}.

\section{The Gelfand transform}

We first reproduce some basic facts of Banach algebra theory related to the Gelfand transform, mostly following \cite[Chaper V, Section 1]{G}.

Let $A$ be a complex commutative Banach algebra with unit. By ${\mathfrak M}_A$ we denote the set of all maximal ideals $m\subset A$ such that $m\neq A$. Every maximal ideal $m\in{\mathfrak M}_A$ is identified with a complex homomorphism $m:A\to{\mathbb C}$ satisfying $m(1)=1$ \cite[Chaper V, Theorem 1.2]{G}. {\it The Gelfand topology} on ${\mathfrak M}_A$ is defined as follows. Given $m_0\in{\mathfrak M}_A$, choose $\varepsilon>0$ and a finite sequence $f_1,\dots,f_n\in A$. Sets of the form
\begin{equation}
\{m\in{\mathfrak M}_A\mid|m(f_i)-m_0(f_i)|<\varepsilon\quad(1\le i\le n)\}
                                \label{7.1}
\end{equation}
constitute a basis of the topology, i.e., open sets in ${\mathfrak M}_A$ are unions of sets of type \eqref{7.1}. Furnished with the Gelfand topology, ${\mathfrak M}_A$ is a compact Hausdorff topological space.

The simplest and most important example is the Banach algebra $C(X)$ of continuous complex-values functions on a compact Hausdorff topological space $X$ with the supremum norm
$\|f\|=\sup_{x\in X}|f(x)|$.
Every point $x\in X$ determines the homomorphism $m_x\in{\mathfrak M}_{C(X)}$ by $m_x(f)=f(x)$ for $f\in C(X)$. Conversely, for every $m\in{\mathfrak M}_{C(X)}$, there exists a point $x\in X$ such that $m=m_x$ \cite[Chapter I, Section 3]{Gm}.

For an arbitrary Banach algebra $A$, the Banach algebra homomorphism
$$
A\to C({\mathfrak M}_A),\ f\mapsto\widehat f
$$
defined by
$
\widehat f(m)=m(f)\ (f\in{\mathfrak M}_A)
$
is called {\it the Gelfand transform}.

The Gelfand topology on ${\mathfrak M}_A$ coincides with the $\ast$-weak topology, i.e., with the weakest topology such that all functions $\widehat f:{\mathfrak M}_A\to{\mathbb C}\ (f\in A)$ are continuous. Moreover, the Gelfand transform does not increase the norm, i.e.,
\begin{equation}
\|\widehat f\|=\sup_{m\in{\mathfrak M}_A}|\widehat f(m)|\le\|f\|\quad (f\in A,m\in{\mathfrak M}_A).
                                \label{7.2}
\end{equation}
A Banach algebra $A$ is said to be {\it a uniform algebra} if there is the equality in \eqref{7.2} for any $f\in A$. By
\cite[Chaper V, Theorem 1.2]{G}, $A$ is a uniform algebra if and only if $\|f^2\|=\|f\|^2$ for all $f\in A$.

Let now $X$ be a compact Hausdorf topological space and let $A$ be a closed subalgebra of $C(X)$ containing constant functions and separating points; such Banach algebras are called {\it uniform algebras on $X$} \cite[Chapter V, Section 1, Example 1]{G}. For such an algebra, the continuous injective map
\begin{equation}
\varepsilon:X\to{\mathfrak M}_A
                                \label{7.3}
\end{equation}
is defined by $\varepsilon(x)=\delta_x$, where $\delta_x$ is the Dirac function supported at the point $x\in X$, i.e., $\delta_x(f)=f(x)$ for $f\in A$. Since $X$ is compact and ${\mathfrak M}_A$ is Hausdorff, $\varepsilon$ is a homeomorphism of $X$ onto some closed subset of ${\mathfrak M}_A$. Unfortunately, there is no common notation for the latter closed subset, let us denote it by $[X]_{{\mathfrak M}_A}$.
If \eqref{7.3} is a homeomorphism of $X$ onto ${\mathfrak M}_A$, i.e., if $[X]_{{\mathfrak M}_A}={\mathfrak M}_A$, we say that {\it $X$ coincides with ${\mathfrak M}_A$} and write $X={\mathfrak M}_A$. At least Garnet \cite{G} uses this terminology. Belishev \cite{Be} says that
{\it $A$ is a generic algebra} if $[X]_{{\mathfrak M}_A}={\mathfrak M}_A$, but I do not like this term. If $X$ coincides with
${\mathfrak M}_A$, then
$$
f=\widehat f\circ\varepsilon^{-1}\quad\mbox{for every}\ f\in A.
$$

Recall that the Banach algebra ${\mathcal A}(M,g)$ was defined in the previous section for a compact oriented flat metric surface $(M,g)$ with a non-empty boundary.

\begin{theorem} \label{Th7.1}
Let $(M,g)$ be a compact flat oriented metric surface with a non-empty boundary. Then $M$ coincides with ${\mathfrak M}_{{\mathcal A}(M,g)}$, i.e., \eqref{7.3} is a homeomorphism of $M$ onto ${\mathfrak M}_A$ for $A={\mathcal A}(M,g)$.
\end{theorem}

We are going to demonstrate that Theorem \ref{Th7.1} follows from the main result of the paper \cite{Ar} by Arens. We first cite Theorem 5.3 of \cite{Ar} in slightly different notations.

Let $M'$ be a Riemann surface. Let $G$ be be an open subset of $M'$, and let $M$ be a compact subset of $M'$ containing $G$. By $CH(M,G)$ we denote the algebra of complex functions defined and continuous on $M$ and moreover holomorphic on $G$. Being furnished with the norm $\|f\|=\sup_{z\in M}|f(z)|$, $CH(M,G)$ is a commutative Banach algebra.

\begin{theorem}[see {\cite[Theorem 5.3]{Ar}}] \label{P7.2}
Every homomorphism $m:CH(M,G)\to{\mathbb C}$ is of the form $m(f)=f(z)$ for some point $z\in M$.
\end{theorem}

\begin{proof}[Proof of Theorem \ref{Th7.1}]
Let $(M,g)$ be a compact flat oriented metric surface with a non-empty boundary. By Proposition \ref{P5.6}, $(M,g)$ can be isometrically embedded into a flat Riemann surface $(M',{\mathcal C}')$ furnished with the metric $g_{{\mathcal C}'}$. Set $G=M\setminus\partial M$. Then $M$ is a compact subset of $M'$, $G$ is an open subset of $M'$, and $G\subset M$. Moreover, the Banach algebra $CH(M,G)$ coincides with
${\mathcal A}(M,g)$. Applying Theorem \ref{P7.2}, we obtain the statement of Theorem \ref{Th7.1}.
\end{proof}

We complete the section with two useful statements that are proved by similar arguments although do not relate to the Gelfand transform.

\begin{proposition} \label{P7.3}
Let $(M,g)$ be a compact oriented flat metric surface with a non-empty boundary. The algebra ${\mathcal A}^\infty(M,g)$ separates points, i.e., for any two different points $x,y\in M$ there exists a function $f\in{\mathcal A}^\infty(M,g)$ such that $f(x)\neq f(y)$.
\end{proposition}

Proposition \ref{P7.3} obviously follows from the stronger statement that will be used later.

\begin{lemma} \label{L7.4}
Let $(M,g)$ be a compact oriented flat metric surface with a non-empty boundary $\Gamma=\partial M$.
For every point $z_0\in M$, there exists a function $f\in{\mathcal A}^\infty(M,g)$ such that $f(z)\neq0$ for all points $M\ni z\neq z_0$ and $z_0$ is a simple zero of $f$. i.e., $f(z_0)=0$ and $df(z_0)\neq0$.
\end{lemma}

\begin{proof}
Let $m$ be the number of boundary circles, see \eqref{2.8}. Gluing a disk ${\D}_j$ to every boundary circle $\Gamma_j$, we obtain a compact oriented two-dimensional manifold $M'$ with no boundary such that
$
M\setminus\Gamma=M'\setminus\bigcup_{j=1}^m{\D}_j.
$
The Riemannian metric $g$ can be extended to a Riemannian metric $g'$ on $M'$. The latter metric determines a
complex structure on $M'$ so that $M'$ becomes a compact Riemann surface.

By the Riemann -- Roch theorem \cite[Theorem 16.9]{F}, there exists a meromorphic function $h$ on $M'$ with one simple zero at $z_0$ and one simple pole at $z_1\in{\D}_1$. The restriction $f=h|_M$ possesses the desired properties.
\end{proof}

We will also need the following statement

\begin{proposition}
Let $(M,g)$ be a compact oriented flat metric surface with a non-empty boundary. The algebra ${\mathcal A}^\infty(M,g)$ is dense in ${\mathcal A}(M,g)$ with respect to the supremum norm.
\end{proposition}

We omit the proof that can be easily obtained on the base of \cite[Theorem 4.1]{Ar}.

\section{Proof of Theorem \ref{Th1.1}}

Recall that we are going to prove Theorem \ref{Th1.1} under the additional assumption: both metric surfaces $(M_1,g_1)$ and $(M_2,g_2)$ are oriented and the isometry $\varphi:\partial M_1\to\partial M_2$ preserves orientations. Under the assumption, Theorem \ref{Th1.1} is equivalent to Theorem \ref{Th3.3}. Thus, we are proving  Theorem \ref{Th3.3}.

Let hypotheses of Theorem \ref{Th3.3} be satisfied.
By Theorem \ref{Th6.2}, the equality $\Lambda_{g_1}=\Lambda_{g_2}$ implies
$
{\mathcal A}(\Gamma,g_1)={\mathcal A}(\Gamma,g_2).
$
Recall that, for $j=1,2$, the trace operator
$$
{\mbox{tr}}_j:{\mathcal A}(M,g_j)\to{\mathcal A}(\Gamma,g_j)
$$
is an isomorphism of Banach algebras preserving norms. Therefore
\begin{equation}
\varphi=\mbox{tr}_1^{-1}\circ\mbox{tr}_2:{\mathcal A}(M,g_2)\to {\mathcal A}(M,g_1)
                                \label{8.1}
\end{equation}
is also an isomorphism of Banach algebras preserving norms. It also preserves constants, i.e., $\varphi c=c$ for
$\mbox{const}=c\in{\mathcal A}(M,g_2)$.

The isomorphism \eqref{8.1} implies the homeomorphism between the spaces of maximal ideals
\begin{equation}
\widetilde\varphi:{\mathfrak M}_{{\mathcal A}(M,g_2)}\to{\mathfrak M}_{{\mathcal A}(M,g_1)},
                                \label{8.2}
\end{equation}
where $\widetilde\varphi(m)=\{\varphi f\mid f\in m\}$ for $m\in{\mathfrak M}_{{\mathcal A}(M,g_2)}$.

Recall that a maximal ideal $m\in{\mathfrak M}_A$ of a Banach algebra $A$ is identified with the homomorphism $m:A\to{\mathbb C}$ of algebras satisfying $m(1)=1$.

Now, looking at \eqref{8.2}, we define the map $\psi:M\to M$ as follows. Given a point $z_1\in M$, let the ideal
$m_1\in{\mathfrak M}_{{\mathcal A}(M,g_1)}$ be defined by $m_1(f)=f(z_1)$ for $f\in{\mathcal A}(M,g_1)$. We set
$m_2=\widetilde\varphi(m_1)\in{\mathfrak M}_{{\mathcal A}(M,g_2)}$. By Theorem \ref{Th7.1}, there exists a unique point $z_2\in M$ such that $m_2(f)=f(z_2)$ for $f\in{\mathcal A}(M,g_2)$. We set $\psi(z_1)=z_2$.

It is clear that $\psi$ is a homeomorphism of $M$ onto itself and that $\psi|_{\partial M}=\mbox{\rm Id}$.

The homeomorphism $\psi$ can be also described as follows. For a point $z\in M$ let $m_{j,z}\in{\mathfrak M}_{{\mathcal A}(M,g_j)}\ (j=1,2)$ be the maximal ideal of functions vanishing at $z$. Then
\begin{equation}
m_{1,z}\circ\varphi=m_{2,\psi(z)}.
                                \label{8.3}
\end{equation}
This equality immediately follows from the definition of $\psi$ and characterizes the homeomorphism $\psi$.

By Proposition \ref{P5.6}, the flat metric $g_j\ (j=1,2)$ determines the flat complex structure ${\mathcal C}_j={\mathcal C}_{g_j}$ on
$\stackrel\circ M=M\setminus\Gamma$. Since $\psi$ is a homeomorphism of $M$ onto itself, the restriction $\psi|_{\stackrel\circ M}$ is also
a homeomorphism of $\stackrel\circ M$ onto itself.

\begin{lemma} \label{L8.1}
Under hypotheses of Theorem \ref{Th3.3}, let the homeomorphism $\psi:M\to M$ be defined by \eqref{8.3}. Then
$\psi|_{\stackrel\circ M}:(\stackrel\circ M,{\mathcal C}_1)\to(\stackrel\circ M,{\mathcal C}_2)$ is a biholomorphism of Riemann surfaces.
\end{lemma}

\begin{proof} (Compare with the proof of \cite[Theorem 1]{Ro}.)
We first prove the following statement:
\begin{equation}
\varphi f=f\circ\psi\quad\mbox{for any}\quad f\in{\mathcal A}(M,g_2).
                                \label{8.4}
\end{equation}
Indeed, choose a point $z_1\in M$ and denote $c=(\varphi f)(z_1)$ and $z_2=\psi(z_1)$. Treating $c$ as the constant function, we obtain with the help of \eqref{8.3}
$$
\varphi f-c\in m_{1,z_1},\quad f-c\in m_{2,z_2}.
$$
Thus the value of $f$ at $z_2$ is also $c$. This proves \eqref{8.4}. Moreover, $\psi$ is uniquely determined by \eqref{8.4}. Indeed, suppose $\psi_1:M\to M$ is another mapping such that $\varphi f=f\circ\psi_1=f\circ\psi$ for any $f\in{\mathcal A}(M,g_2)$. If there were a point $z$ where $\psi(z)\neq\psi_1(z)$ we could construct a function $f\in{\mathcal A}(M,g_2)$ with different values at these points and arrive at a contradiction.

Fix a point $z\in\stackrel\circ M$. By Lemma \ref{L7.4}, there exists a function $f\in{\mathcal A}(M,g_2)$ with a simple pole at $\psi(z)$. Set $g=\varphi f$. Then there is a neighborhood $U\subset\stackrel\circ M$ of $\psi(z)$ in which the function $f$ is univalent. Take a neighborhood $V\subset\stackrel\circ M$ of $z$ such that $V\subset\psi^{-1}(U)$ and $g(V)\subset f(U)$. Then in $V$ we have the representation $f^{-1}\circ g$ for the mapping $\psi$, and hence $\psi|_{\stackrel\circ M}:(\stackrel\circ M,{\mathcal C}_1)\to(\stackrel\circ M,{\mathcal C}_2)$ is an analytic mapping.

Let $\varphi^{-1}$ be the inverse isomorphism of \eqref{8.2}. Let $\psi$ and $\psi'$ be homeomorphisms of $M$ onto itself associated with $\varphi$ and $\varphi^{-1}$ respectively. Then $\psi\circ\psi'$ and $\psi\circ\psi'$ are homeomorphisms of $M$ onto itself which induce the identity isomorphisms of ${\mathcal A}(M,g_1)$ and ${\mathcal A}(M,g_2)$ respectively. By uniqueness we see that $\psi\circ\psi'$ and $\psi\circ\psi'$ must be identity maps. Thus $\psi'=\psi^{-1}$, and $\psi|_{\stackrel\circ M}$ is a biholomorphism between
$(\stackrel\circ M,{\mathcal C}_1)$ and $(\stackrel\circ M,{\mathcal C}_2)$.
\end{proof}

Finally, we combine Proposition \ref{P4.1} and Lemma \ref{L8.1}. Since
$\psi|_{\stackrel\circ M}:(\stackrel\circ M,{\mathcal C}_1)\to(\stackrel\circ M,{\mathcal C}_2)$ is a biholomorphism of Riemann surfaces,
$\psi|_{\stackrel\circ M}:\stackrel\circ M\to \stackrel\circ M$ is a diffeomorphism. We have seen that the homeomorphism $\psi:M\to M$ fixes the boundary, $\psi|_\Gamma=\mbox{Id}$. Since ${\mathcal C}_j={\mathcal C}_{g_j}$, the statement of Lemma \ref{L8.1} means that
$\psi|_{\stackrel\circ M}:(\stackrel\circ M,g_1)\to(\stackrel\circ M,g_2)$ is a conformal map between flat metric surfaces. Thus, $\psi$ satisfies all hypotheses of Proposition \ref{P4.1}. Applying the proposition, we infer that
$\psi:(M,g_1)\to(M,g_2)$ is an isometry. This completes the proof of Theorem \ref{Th3.3}.

{\bf Remark.} Belishev does not prove Theorem \ref{Th3.3} in \cite{Be}. There is only one sentence that can be related to Lemma \ref{L8.1}.
Namely, we read on page 180 of \cite{Be}: ``Step 4. Using functions $w\in{\mathcal A}(\Omega)$ as local homeomorphisms $\Omega\to{\mathbb C}$ recover on $\Omega$ the complex structure.'' The sentence is hard understandable and proves nothing.

\end{document}